\documentclass[11pt]{amsart}   	% use "amsart" instead of "article" for AMSLaTeX format
\usepackage{accents}	
\usepackage{adjustbox}
\usepackage{amssymb, latexsym}
\usepackage{amsmath}
\usepackage[toc,page]{appendix}
\usepackage{braket}
\usepackage{breqn}
\usepackage{caption}
\usepackage{color}
\usepackage{comment}
\usepackage{dsfont}
\usepackage{esint}
\usepackage[shortlabels]{enumitem}
\usepackage[T1]{fontenc}
\usepackage{geometry}         
\usepackage{graphicx}
\usepackage{lipsum}
\usepackage{hyperref}
\usepackage{latexsym}
\usepackage{mathrsfs}
\usepackage{mathtools}
\usepackage{subcaption}
\usepackage{float}
\restylefloat{table}
\usepackage{listings}
\usepackage{wasysym}
\usepackage[dvipsnames]{xcolor} %can be used to highlight text

\usepackage{times}
\usepackage{url} %does nice formatting of urls
\usepackage{tikz}
\usetikzlibrary{cd}
\usetikzlibrary{positioning}

\theoremstyle{plain}
\numberwithin{equation}{section}
\newtheorem{thm}{Theorem}[section]
\newtheorem{theorem}[thm]{Theorem}
\newtheorem{lemma}[thm]{Lemma}

\newtheorem{assumption}[thm]{Assumption}

\newtheorem{definition}[thm]{Definition}

\newtheorem{proposition}[thm]{Proposition}

\newtheorem{remark}[thm]{Remark}

\def\al{\alpha}

\def\be{\beta}

\def\de{\delta}

\def\ep{\epsilon}
\def\ga{\gamma}
\def\Ga{\Gamma}

\def\ka{\kappa}

\def\Om{\Omega}

\def\sig{\sigma}

\def\th{\theta}

\def\N{\mathbb{N}}

\def\R{\mathbb{R}}

\def\grad{\nabla}
\def\lang{\langle}

\def\rang{\rangle}

\def\Lin{\mathcal{L}}

\def\exp{\text{exp}}

\newcommand\beq{\begin{equation}}
\newcommand{\bburl}[1]{\textcolor{blue}{\url{#1}}}
\newcommand\eeq{\end{equation}}
\newcommand\bea{\begin{eqnarray}}
\newcommand\eea{\end{xq}}
\newcommand\bi{\begin{itemize}}
\newcommand\ei{\end{itemize}}
\newcommand\ben{\begin{enumerate}}
\newcommand\een{\end{enumerate}}

\newcommand{\LC}[1]{{\color{purple} {\bf Loic:} [#1] }}
\newcommand{\JMS}[1]{{\color{orange} {\bf J. Siktar:} [#1] }}
\newcommand{\GJ}[1]{{\color{red} {\bf Gabriela:} [#1] }}

\everymath{\displaystyle}

\title{A Mountain-Pass Algorithm for Nonlocal Problems with Super-quadratic Nonlinearities}

\author{Loic Cappanera, Gabriela Jaramillo, Joshua M. Siktar}
\date{\today}							% Activate to display a given date or no date

\begin{document}

\begin{abstract}
% \GJ{Titles:

% 1. Multiplicity Results for Solutions to Diffusion Problems with Super-quadratic Nonlinearities

% 2. 
% }

 In this paper we consider a nonlinear equation $-\Lin u(x) = f(x, u(x))$ with a super-quadratic nonlinearity, $f$, and a nonlocal operator, $\Lin$, generated by a special class of radially symmetric $L^1$ convolution kernels with finite second moments. The assumptions on this operator are mild and allow for a variety of kernels used in biological and physical applications, including kernels with algebraic decay and sign changing kernels. 
   
    Using the strong nonlinearities present in the equation, we prove the existence of nontrivial solutions using the classical Mountain Pass Theorem, a central result in minimax theory that equates solutions of our equation to critical points of a corresponding energy functional. This existence result holds with both homogeneous nonlocal Dirichlet and nonlocal Neumann boundary conditions. We supplement these theoretical results with numerical simulations for various nonlinearities with odd maximal degree in the unknown $u$. The numerical scheme exploits the resulting energy landscape which allows one to adapt a gradient descent algorithm.
\end{abstract}

\maketitle

\section{Introduction}

In this paper we consider nonlocal equations of the form
\begin{equation}\label{e:nonlocal_eq}
\begin{array}{c c c c }
-\mathcal{L} u(x) \ & = \ f(x,u) & \quad  x \in & \Omega \subset \mathbb{R}^n \\ %\quad n=2,\\
%u & = 0 & \quad x \in & \Omega^c = \mathbb{R}^n \setminus \Omega
\end{array}
\end{equation}
where the function $f(x,u)$ is nonlinear in $u$ and the map $\mathcal{L}$ is a nonlocal operator of the form
\begin{equation}\label{e:conv_operator}
\mathcal{L} u(x) = \int_{\R^n} (u(y) - u(x) ) \gamma(x,y) \; dy.
\end{equation}
We assume that the convolution kernel,  $\gamma(x,y) = \gamma(|x-y|)$, is a radially symmetric function living in $L^1(\R^n)$ and satisfies appropriate conditions, which we make more precise in what follows. In addition, we take $\Omega$ to be a bounded domain with smooth boundary, $\partial \Omega$, and consider homogeneous nonlocal Dirichlet and Neumann boundary constraints given, respectively, by
\begin{equation}\label{e:dirichlet}
u \ = \ 0 \quad x \in \Omega^c \ = \ \mathbb{R}^n \setminus \Omega, %\quad n =2,
\end{equation}
\begin{equation}\label{e:neumann}
\mathcal{L} u \ = \ 0 \quad x \in \Omega^c \ = \ \mathbb{R}^n \setminus \Omega.
\end{equation}
Our goals here are two-fold: one is to show that \eqref{e:nonlocal_eq} admits non-trivial solutions using the Mountain-Pass Theorem,  and the second is to present an efficient numerical scheme for finding such solutions based on the particular geometry of the energy functional derived from \eqref{e:nonlocal_eq}.

Our motivation for considering equations of the form \eqref{e:nonlocal_eq} comes from biological and physical applications where long-range dispersal events, or nonlocal interactions, give rise to integral operators. For example, various population \cite{cosner2012, hutson2003, lutscher2005}, predator-prey \cite{hao2021, sherratt2016,wu2023}, epidemic \cite{ying2013, yang2015, zhao2018}, and vegetation models \cite{baudena2013, eigentler2018, thompson2009} use integral operators like the one given in \eqref{e:conv_operator} to describe how individuals or plant seeds disperse. Solutions to systems of equations of the form \eqref{e:nonlocal_eq} then correspond to time independent solutions, or steady states, for these models. This type of operator also appears in descriptions of excitable systems, including cardiac and brain tissue. In this case, maps like \eqref{e:conv_operator} are used to describe the effects of tissue inhomogeneities \cite{bueno-orovio2014}, or long-range excitatory and inhibitory interactions between neurons \cite{faye2015, ermentrout2001}. Similarly, models for phase transitions that account for long-range interaction between phases result in Landau-type functionals where the Dirichlet, or gradient, energy is replaced by a convolution \cite{bates2006, du2024}.  Interest in this type of functional then stems from its ability to produce  solutions with sharp interfaces that resemble more closely results seen in applications \cite{burkovska2023, olena2021}. Calculating the $L^2$ gradient flow of this energy then results in a nonlocal Allen-Cahn equation whose steady states then correspond to solutions of \eqref{e:nonlocal_eq}. 

%The corresponding equations for these nonlocal phase field models are derived as the gradient flow of this functional, giving a nonlocal Allen-Cahn equation in the case of $L^2$ gradient flow. Again we have that steady states of this equation then correspond to solutions of \eqref{e:nonlocal_eq}. 
% which are capable of producing

In contrast to other works that model the operator $\mathcal{L}$ as a fractional Laplacian \cite{chammem2023combined, dipierro2017fractional, gu2018infinitely, servadei2012mountain}, the analytical results presented in Sections \ref{s:mountain-pass} and \ref{Sec: Special} are based on integral operators \eqref{e:conv_operator}  generated by non-singular, radially symmetric kernels having finite second moments. In addition, we assume that these operators can be split into the sum of a diffusive operator and a small perturbation (see Section \ref{Sec: Setup} for a precise formulation of these assumptions). In particular, we allow $\gamma$ to take on negative values and consider both thin-tailed (exponential decay) and fat-tailed (algebraic decay) kernels. As a result,  equation \eqref{e:nonlocal_eq} is able to encapsulate models from various applications. For instance, sign changing kernels are relevant in neural field models where, as mentioned above, they describe interactions between neurons. In this case, negative values of $\gamma$ then correspond to inhibitory connections. On the other hand, non-negative kernels arise  naturally in ecological applications \cite{bullock2017}, where the level of decay of $\gamma$ at infinity can distinguish behavior from different plant species and also contribute to different phenomena. Indeed, it has been shown that while kernels with exponential decay can produce population fronts traveling with constant speed, algebraically decaying kernels give rise to accelerating fronts \cite{henderson2018}.

To prove the existence of nontrivial solutions for equation \eqref{e:nonlocal_eq} we adapt the classical Mountain-Pass Theorem to our nonlocal setting. This theorem, originally proven in \cite{rabinowitz1982mountain}, is a central part of minimax theory, as it allows one to prove existence of critical points of an energy functional. These results can then be applied to obtain  existence of non-trivial solutions to the corresponding Euler-Lagrange equations. The  theorem requires energy functionals to possess a particular geometric structure whereby their graph exhibits a "valley" or saddle structure. In particular, the energy functional must be positive for all points at a fixed distance away from the origin, and be negative at some point farther away. The theorem also requires a compactness condition  to hold on sequences of points for which the derivative of the energy functional vanishes in the limit. As we will show, the growth conditions on the nonlinearity $f$ play a crucial role in verifying these assumptions.

While the Mountain-Pass Theorem was originally used to prove multiplicity of solutions for elliptic partial differential equations (see for instance \cite{choi1993semiwave, choi1993mountain, clark1972variant, de2009multiple, degiovanni2022variational,  kajikiya2005critical, rabinowitz1986minimax, struwe2000variational}), more recent applications of the theorem and its generalization, the Linking Theorem, have been extended to equations involving the fractional Laplacian or other nonlocalities \cite{bisci2014mountain, bisci2014three, bisci2016variational, bors2014application, chammem2023combined, dipierro2017fractional, gu2018infinitely, maione2023variational, servadei2012mountain, servadei2013variational}. As in the case of fractional operators, we find that the long range interactions inherent to nonlocal models do not significantly affect the verification of the conditions for the Mountain-Pass Theorem; having proper structure on the underlying function spaces, including suitable continuous embeddings, is more important. 
Because the nonlocal operators considered here have $L^2(\Omega)$ as a natural domain, we do not obtain the aforementioned necessary compactness condition. To overcome this difficulty, we reformulate the operators so that they generate bilinear forms with domain $H^1(\Omega)\times H^1(\Omega)$. To our knowledge, this work is the first to apply minimax theory to such problems.

To accomplish our second goal and obtain an efficient numerical scheme to solve equation \eqref{e:nonlocal_eq}, we follow the approach taken in \cite{bailova2021mountain, bailova2023new,  choi1993semiwave, choi1993mountain}. 
Recasting the problem as finding critical points of an energy, we exploit a stricter geometric assumption (expressed as a condition on the nonlinearity $f(x,u)$) to first find an explicit expression for the maximum of this energy along any radial direction. This allows one to obtain a good first guess for a gradient descent scheme, which then updates the direction of descent (in the ``angular" direction) by solving a nonlocal linear equation. Our results show that the proposed scheme performs better than a standard Newton algorithm, which while converging cannot find a true solution to the nonlocal equation.

We finish this section with a brief outline of our paper. In Section \ref{Sec: Setup} we state our main assumptions for the kernel $\gamma$, describe properties of the operator $\mathcal{L}$ and its corresponding bilinear form, and summarize necessary conditions related to the nonlinear term $f(x,u)$. In Section \ref{s:mountain-pass} we adapt the Mountain-Pass Theorem to our nonlocal setting. In Section \ref{Sec: Special} we provide an alternative proof for existence of non-trivial solutions based on a slightly more restrictive set of assumptions on the nonlinearity $f$ which follows the results from \cite{bailova2021mountain}. This analysis is then the basis for our numerical scheme, which we present in Section \ref{s:numerics}. We conclude the paper with a discussion and summary of future work.

%%%%%%%%

\section{Background and Preliminaries}\label{Sec: Setup}

In this section we consider operators of the form
\begin{equation}\label{e:nonlocal_operator}
\mathcal{L} u \ = \  \int_{\R^n} (u(y) - u(x) ) \gamma(|x-y|) \; dy, \qquad x \in \Omega \subset \R^n,
\end{equation}
and as stated in the introduction, we assume $\Omega$ is a bounded domain with smooth boundary. In addition, we take convolution kernels  $\gamma: \R^n \longrightarrow \R$ that are at minimum symmetric functions in $L^1(\R^n)$. 
Throughout the paper, we will further assume that these maps can be split into the sum of a diffusive operator, $\mathcal{L}_e$, and a small perturbation $\mathcal{L}_s$.
In order to precisely state this assumption, in the next two subsections we first define diffusive operators and then state properties of their corresponding bilinear forms. In Subsection \ref{ss:properties_nonlocal_op} we then characterize what we mean when we say that $\mathcal{L}_s$ is a small perturbation. Examples of the type of operators considered here are listed at the end of this subsection. Finally, Subsection \ref{ss:nonlinearity} gives assumptions on the nonlinearities $f(x,u)$ considered in this paper.

%%%%%%%%%%%%%%%%%%

\subsection{Properties of Nonlocal Diffusive Operators}\label{Subsec: PropL}

Throughout the paper, we make the following assumption on the diffusive operator $\mathcal{L}_e$.

\begin{assumption}\label{a:operatorL}
There is a constant $\xi_0>0$ such that the operator $\mathcal{L}_e $ acting on functions $L^2(\R^n) $  has Fourier symbol $L(\xi) = L(|\xi|)$ which is uniformly bounded and analytic on a strip $| \mathrm{Re}(\xi)| < \xi_0$ of the complex plane. Moreover, the symbol $L(\xi)$ has a Taylor expansion
\[ L(\xi) = - |\xi|^2 + \mathrm{O}(|\xi|^4) \quad \text{as} \quad |\xi| \to 0.\]
\end{assumption}

Assumption \ref{a:operatorL} guarantees that the map $\mathcal{L}_e$ can be written as the composition of an invertible map and the Laplace operator. This result is summarized in the following lemma, which is a consequence of the removable singularity theorem; see \cite[Lemma 3.9]{jaramillo2019effect} for a proof.
\begin{lemma}\label{l:decomposeL}
Let $L(\xi)$ denote the multiplication operator satisfying Assumption \ref{a:operatorL} for some $\xi_0 >0$. Then $L(\xi)$ admits the following decomposition:
\begin{equation}\label{Eq: SpectralMult}
L(\xi) \ = \ M_L(\xi) L_{NF}(\xi) \ = \ L_{NF}(\xi) M_R(\xi),
\end{equation}
where $L_{NF}(\xi) = - |\xi|^2/ (1+ |\xi|^2)$, while $M_{L/R}(\xi)$ and their inverses are analytic and uniformly bounded on $| \mathrm{Re}(\xi)| < \xi_1$ for some $\xi_1 < \xi_0$. 
\end{lemma}

We use the above results  to make one further assumption on the symbol $L$.

\begin{assumption}\label{a:decomposeL}
Let $L(\xi)$ denote the multiplication operator satisfying Assumption \ref{a:operatorL} and consider its decomposition \eqref{Eq: SpectralMult}, as stated in Lemma \ref{l:decomposeL}. We assume that $M_R(\xi) = M_L(\xi)$ is such that 
 \[ L(\xi) = -|\xi|^2 (1 + L_1(\xi))\]
 with $L_1(\xi)$ satisfying Assumption \ref{a:operatorL} .
\end{assumption}

Notice that Assumption \ref{a:operatorL} implies that we can write the operator $\mathcal{L}_e $ as in \eqref{e:nonlocal_operator},
% \begin{equation}\label{e:integralL}
%  \mathcal{L}u \ = \ d\int_{\R^n} (u(y) - u(x) ) \gamma(|x-y|)\;dy,
%  \end{equation}
%  \JMS{I think this equation is redundant with the previous one, same as $\mathcal{L} u$ from before}
where the convolution kernel $\gamma(|x-y|) = \gamma(z)$  is such that:
\begin{enumerate}[label=\textbf{P\arabic*}]
\item\label{Q1} $\gamma: \R_+ \longrightarrow \R$ is exponentially decaying,
\item\label{Q2} $\int_{\mathbb{R}^n} | \gamma(|x|)| \; dx = 1$,
\item \label{Q3}$\int_{\mathbb{R}^n} x^2 \gamma(|x|) \; dx \neq 0$.
\end{enumerate}
However,  this is not an equivalence relation. The above conditions on $\gamma$ are not sufficient to obtain an operator $\mathcal{L}_e$ satisfying assumptions \ref{a:decomposeL}.

For illustrative purposes, we give examples of Fourier symbols satisfying the above assumptions. We also give an explicit expression for the associated convolution kernel $\gamma(z)$.

{\bf Example 1:} 
\[ L(\xi) = \frac{1}{b^2} \left[ -1 + \frac{1}{1+ b^2 |\xi|^2} \right] = \frac{-|\xi|^2}{1 +b^2  |\xi|^2} = -|\xi|^2 \left[ 1 + \underbrace{\left( \frac{b^2(1+|\xi|^2)}{1+b^2 |\xi|^2} \right) \cdot   \frac{-|\xi|^2}{1 + |\xi|^2} }_{L_1(\xi)} \right]. \]
Tying this symbol back to the definition of the operator $\mathcal{L}$ given in \eqref{e:nonlocal_operator} we see that the corresponding convolution kernel $\gamma(z)$, which has Fourier symbol $\hat{\gamma}(\xi) = (1+b^2|\xi|^2)^{-1}$, is the Green's function for the  operator $( \mathrm{Id} -b^2 \Delta)$, where $b\in \R$. Then,
\begin{itemize}
\item $\gamma(z) = \frac{1}{2b} \exp(-|z|/b)$ in dimension $n=1$, and
%\JMS{I don't think it's specified what $b$ is?}
\item $\gamma(z) = \frac{1}{b\pi^2} K_0(|z|/b)$  in dimension $n=2$, where $K_0$ is the modified Bessel function of degree zero.
\end{itemize}

{\bf Example 2:}
\[ L(\xi) =  -1 + \exp(-b^2 |\xi|^2/4)  = -\frac{b^2}{4} |\xi|^2 \left[ 1+  \underbrace{\frac{-1 + \frac{b^2}{4} |\xi|^2 + \exp(-b^2 |\xi|^2/4)}{ -\frac{b^2}{4} |\xi|^2}}_{L_1(\xi)} \right] \]
In this case, the convolution kernel that gives rise to this Fourier symbol is
\begin{itemize}
\item $\gamma(z) = \frac{1}{\pi^{n/2} b^n} \exp( -|z|^2/b^2)  $, with $b>0$ and for all dimensions $n$.
\end{itemize}

%%%%%%%%%%%%%%%%%%%%

\subsection{Bilinear Forms of Nonlocal Diffusive Operators and their Properties}\label{Subsec: BilinearFormB}
We start by considering the nonlocal Dirichlet problem
\begin{equation}\label{e:nonlocal_D}
\begin{split}
-\mathcal{L}_e u & = f(x,u) \qquad  x \in  \Omega \subset \mathbb{R}^n,\\
u & = 0 \hspace{1.75cm} x \in \Omega^c = \R^n \setminus \Omega.
\end{split}
\end{equation}
where $\mathcal{L}_e$, as in Subsection \ref{Subsec: PropL}, represents a nonlocal operator of diffusive type. For convenience, we drop the subscript $e$ in the rest of this subsection. We view the associated bilinear form $B_D[\cdot, \cdot]$, defined below, as a functional with domain $X^1_D  \times X^1_D $ where 
\[ X^1_D = \{ u \in H^1(\Omega) \mid \quad u =0, \quad x \in \Omega^c\}.\]
It is straightforward to check that this space is equivalent to $H^1(\Omega)$ and so we equip it with the $H^1$ norm
\[ \|u\|_{H^1(\Omega)} = \| u\|_{L^2(\Omega)} + \| \nabla u\|_{L^2(\Omega)}.\]
We then have the following definition.

\begin{definition}\label{d:bilinearform}
Let $\gamma: \R^n \longrightarrow \R$ be such that the associated integral operator $\mathcal{L}$, given by 
\eqref{e:nonlocal_operator}, satisfies Assumptions \ref{a:operatorL} and \ref{a:decomposeL}. 
Define then the bilinear form $B_D: X^1_D \times X^1_D \longrightarrow \R$ by the expression
\[ B_D[u,\phi] = \frac{1}{2} \iint_{\R^n\times \R^n }  (u(y) - u(x)) \gamma(|x-y|) (\phi(y) -\phi(x)) \;dy\;dx.\]
\end{definition}

It is easy to see that the operator $\Lin$ and the kernel $\ga$ adhere to the  nonlocal Green's identity%\JMS{Loic and Gabriela: here is where I decided to add the NL Green's Identity, but feel free to move}
\begin{equation}\label{Eq: NLGreen}
    \int_{\R^n}\Lin u(x) v(x) dx \ = \ -\frac{1}{2}\iint_{\R^{2n}}(u(y) - u(x))\ga(x, y)(v(y) - v(x))dydx 
\end{equation}
for any $u, v \in L^2(\R^n)$. Moreover, thanks to our definition for the space $X^1_D$ and the radial symmetry of the kernel $\gamma$, we may write
\begin{equation}\label{Eq:BilinearToL_D}
B_D[u,\phi] \ = \ (-\mathcal{L}u, \phi)_{L^2(\Omega)} \ = \ (-\mathcal{L}u, \phi)_{L^2(\R^n)}.
\end{equation}
The bilinear form (which is clearly symmetric) also satisfies the following properties.

\begin{proposition}\label{p:bilinearform1}
Let $B_D: X^1_D \times X^1_D \longrightarrow \R $ be defined as in Definition \ref{d:bilinearform} , then there exist positive constants $C_i$ and $\beta_i$, $i =1,2,$ such that
\begin{enumerate}[label=\textbf{D\arabic*}]
\item\label{D1} $ B_D[u,\phi] \leq C_1 \| u\|_{L^2(\Omega)} \| \phi\|_{L^2(\Omega)}$
\item\label{D2} $ B_D[u,u] \geq \beta_1 \| u\|^2_{L^2(\Omega)}$
\item\label{D3} $ B_D[u,\phi] \leq C_2 \| u\|_{H^1(\Omega)} \| \phi\|_{H^1(\Omega)}$
\item\label{D4} $ B_D[u,u] \geq \beta_2 \| u\|^2_{H^1(\Omega)}$
\end{enumerate}
for all $u, \phi \in X^1_D$.
\end{proposition}

\begin{proof}
The properties \ref{D1} and \ref{D2} follow from the results in \cite{cappanera2024analysis}, which consider the same bilinear form but with domain $V_D \times V_D$, where
 \begin{equation}\label{Eq: VDDef}
 V_D \ = \ \{ u \in L^2(\R^n) \mid \quad u =0, \quad x \in \Omega^c\}.
 \end{equation}

To prove the properties \ref{D3} and \ref{D4} we use Assumption \ref{a:decomposeL} to write
\begin{equation}\label{Eq: DecomposeL1}
\begin{aligned}
B_D[u,\phi] = & - (\mathcal{L}u, \phi)_{L^2(\R^n)}\\
= & - (\Delta u - \Delta \mathcal{L}_1u, \phi)_{L^2(\R^n)}\\
= & \int_{\R^n} \nabla u \cdot \nabla \phi \;dx   - \int_{\R^n} \nabla \mathcal{L}_1 u \cdot \nabla \phi \;dx\\
= &   \int_{\R^n} \nabla u \cdot \nabla \phi \;dx - \int_{\R^n}  \mathcal{L}_1 \nabla u \cdot \nabla \phi \;dx\\
= &   \int_{\Omega} \nabla u \cdot \nabla \phi \;dx + B_1[ \nabla u , \nabla \phi ] .
\end{aligned}
\end{equation}
Here the bilinear form $B_1: V_D \times V_D \longrightarrow \R$ is generated by the operator $\mathcal{L}_1$ with Fourier symbol $L_1(\xi)$ satisfying Assumption \ref{a:operatorL}. Applying the results from  \cite[Lemma 2.3]{cappanera2024analysis} shows that $B_1$ is coercive and bounded. Hence,
\begin{equation}\label{Eq: BDLB}
\begin{aligned}
B_D[u,\phi] & \leq  \| \nabla u\|_{L^2(\Omega)}  \| \nabla \phi \|_{L^2(\Omega)}  + \tilde{C}_1 \| \nabla u\|_{L^2(\Omega)}  \| \nabla \phi \|_{L^2(\Omega)} \\
& \leq  (1+ \tilde{C}_1)   \| \nabla u\|_{L^2(\Omega)}  \| \nabla \phi \|_{L^2(\Omega)}. 
\end{aligned}
\end{equation}

Similarly,
\begin{equation}\label{Eq: BDUB}
\begin{aligned}
B_D[u,\phi] \ & \geq \ \| \nabla u\|^2_{L^2(\Omega)} + \tilde{\beta}_1 \| \nabla u\|^2_{L^2(\Omega)}  \\
& \ \geq \ (1+ \tilde{\beta}_1)   \| \nabla u\|^2_{L^2(\Omega)}. 
\end{aligned}
\end{equation}
The results of the lemma then follow by adding \eqref{Eq: BDLB} and \eqref{Eq: BDUB} with those from \ref{D1} and \ref{D2}.

\end{proof}

Next, we consider the Neumann problem,
\begin{equation}\label{e:nonlocal_N}
\begin{split}
-\mathcal{L} u \ & = \ f(x,u) \qquad  x \in  \Omega \subset \mathbb{R}^n\\
-\mathcal{L}u \ & = \ 0 \hspace{1.75cm} x \in \Omega^c = \R^n \setminus \Omega.
\end{split}
\end{equation}
To define the corresponding bilinear form we use the space
\[ X^k_N = \{ u \in H^k(\R^n) : \mathcal{L}u = 0, \quad  x \in \Omega^c\}\]
equipped with the norm,
$\|u\|_{X^k_N} = \| u\|_{H^k(\R^n)} + |u|_{B_N}$, where
\[|u|^2_{B_N}  = \frac{1}{2} \iint_{\R^n \times \R^n} (u(y) -u(x))^2 \gamma(x,y) \;dx\;dy.\]
\begin{remark}\label{r:propertyX_N}
Notice that because the operator $\mathcal{L}$ is a convolution operator, given $\mathcal{L}u =0$ on $\Omega^c$, we also obtain that $\mathcal{L} \nabla u =  \nabla \mathcal{L}u =  0$ on $\Omega^c$. 
\end{remark}

The following Lemma, which follows from the analysis presented in \cite{olena2021}, shows that space $X^0_N$ is equivalent to $L^2(\Omega)$. 

\begin{lemma}\label{l:equivalent_norms}
Let $\gamma: \R^n \longrightarrow (0,\infty)$ be such that it satisfies conditions \ref{Q1}, \ref{Q2}, and \ref{Q3}.
Then, there exist positive constants $c$ and $C$ such that for any $u \in X^0_N$,
\[ c\|u\|_{L^2(\Omega)} \ \leq \ \| u\|_{X^0_N} \ \leq \ C\|u\|_{L^2(\Omega)}.\]
\end{lemma}

A proof can be found in Section \ref{Subsec: posKer}. It is based on the fact that the exterior problem
\begin{equation}\label{e:exterior_problem}
\begin{split}
\mathcal{L}u \ =& \ 0 \quad x \in \Omega^c\\
u \ = & \ g \ \quad x \in \Omega,
\end{split}
\end{equation}
where $g \in L^2(\Omega)$, has a unique weak solution. This can be interpreted as saying that the condition $\mathcal{L} u=0$ uniquely extends the values of $u$ to the exterior domain $\Omega^c$.

Having defined the required spaces we now consider the following bilinear form.

\begin{definition}\label{d:bilinearform_N}
Let $\gamma: \R^n \longrightarrow (0,\infty)$ be such that the associated integral operator $\mathcal{L}$, given by 
\eqref{e:nonlocal_operator}, satisfies Assumptions \ref{a:operatorL} and \ref{a:decomposeL}. 
Define then the bilinear form $B: X^1_N \times X^1_N \longrightarrow \R$ by the expression
\[ B_N[u,\phi] \ = \ \frac{1}{2} \iint_{\R^n \times \R^n}  (u(y) - u(x)) \gamma(|x-y|) (\phi(y) -\phi(x)) \;dy\;dx.\]
\end{definition}

The definition of the space $X^1_N$ and  the radial symmetry of the kernel $\gamma$, guarantee that
\begin{equation}\label{Eq:BilinearToL_N}
B_N[u,\phi] \ = \ (-\mathcal{L}u, \phi)_{L^2(\Omega)} \ = \ (-\mathcal{L}u, \phi)_{L^2(\R^n)} 
\end{equation}
We also readily obtain that $B_N$ is symmetric.

The next proposition shows that the above bilinear form is bounded and coercive. A version of this result 
can be found in \cite{cappanera2024analysis}, where the base space is taken to be $X^0_N$ and the convolution kernel $\gamma(\cdot)$ is assumed to have compact support and finite second moment. The results from this section show that these same properties can be extended to bilinear forms defined by nonlocal operators satisfying assumptions \ref{a:operatorL} and \ref{a:decomposeL}, and which in addition have a strictly positive kernel. 

\begin{proposition}\label{p:bilinear_N}
Let $B_N: X^1_N \times X^1_N \longrightarrow \R $ be given as in Definition\ref{d:bilinearform_N}. Then, there exist positive constants $C_i$ and $\beta_i$, $i=1,2$, such that
\begin{enumerate}[label=\textbf{N\arabic*}]
\item\label{N1} $ B_N[u,\phi] \leq C_1 \| u\|_{L^2(\Omega)} \| \phi\|_{L^2(\Omega)}$
\item\label{N2} $ B_N[u,u] \geq \beta_1 \| u\|^2_{L^2(\Omega)}$
\item\label{N3} $ B_N[u,\phi] \leq C_2 \| u\|_{H^1(\Omega)} \| \phi\|_{H^1(\Omega)}$
\item\label{N4} $ B_N[u,u] \geq \beta_2 \| u\|^2_{H^1(\Omega)}$
\end{enumerate}
for all $u, \phi \in X^1_N$.
\end{proposition}
\begin{proof}
 We first look at $B_N: X^0_N \times X^0_N \longrightarrow \R$.
It is straightforward to show that there is a constant $C_1>0$ such that,
\[ B_N[u,\phi] \leq C_1 \| u\|_{L^2(\Omega)} \| \phi\|_{L^2(\Omega)}.\]
To prove the coercivity of $B_N$ we use Lemma \ref{l:Poincare}, shown below, which establishes the Poincar{\'e} inequality
\[ \|u\|^2_{X^0_N} \leq \beta B_N[u,u].\]
Property \ref{N2} then follows from Lemma \ref{l:equivalent_norms} which shows that the norms $\| \cdot\|_{X^0_N}$ and $\| \cdot\|_{L^2(\Omega)}$ are equivalent.

As in Proposition \ref{p:bilinearform1} we use Assumption \ref{a:decomposeL} and write
\begin{equation}\label{Eq: DecomposeL1_2}
\begin{aligned}
B_N[u,\phi] = & - (\mathcal{L}u, \phi)_{L^2(\R^n)}\\
= & - (\Delta u - \Delta \mathcal{L}_1u, \phi)_{L^2(\R^n)}\\
= &   \int_{\Omega} \nabla u \cdot \nabla \phi \;dx + B_1[ \nabla u , \nabla \phi ]. 
\end{aligned}
\end{equation}
Because the operator $\mathcal{L}_1$ satisfies hypothesis \ref{a:operatorL}, by our previous argument the corresponding bilinear form $B_1: X^0_N \times X^0_N \longrightarrow \R$ is bounded and coercive. Since by Remark \ref{r:propertyX_N},  $\nabla u$ is also in the  space $X^0_N$, we find that
 \begin{equation*}
\begin{aligned}
B_N[u,\phi] & \leq  \| \nabla u\|_{L^2(\Omega)}  \| \nabla \phi \|_{L^2(\Omega)}  + \tilde{C}_1 \| \nabla u\|_{L^2(\Omega)}  \| \nabla \phi \|_{L^2(\Omega)} \\
& \leq  (1+ \tilde{C}_1)   \| \nabla u\|_{L^2(\Omega)}  \| \nabla \phi \|_{L^2(\Omega)},
\end{aligned}
\end{equation*}
and
\begin{equation*}
\begin{aligned}
B_N[u,\phi] \ & \geq \ \| \nabla u\|^2_{L^2(\Omega)} + \tilde{\beta}_1 \| \nabla u\|^2_{L^2(\Omega)}  \\
& \ \geq \ (1+ \tilde{\beta}_1)   \| \nabla u\|^2_{L^2(\Omega)}.
\end{aligned}
\end{equation*}
The properties \ref{N3} and \ref{N4} then follow.
\end{proof}

The next lemma establishes a Poincar{\'e} inequality for the bilinear form $B_N: X^0_N \times X^0_N \longrightarrow \R$.
A proof of this result can be found in Section \ref{Subsec: posKer}. A similar result can be found in \cite{mengesha2013analysis} for bilinear forms defined by compactly supported convolution kernels $\gamma$.

\begin{lemma}\label{l:Poincare}
Consider the bilinear form $B_N: X^0_N \times X^0_N \longrightarrow \R$ as given by Definition  \ref{d:bilinearform_N}. Then, there exists a positive constant $\beta$, such that
\[ \| u\|^2_{L^2(\R^n)} \ \leq \ \beta B_N[u,u].\]
\end{lemma}

%%%%%%%%%%%%%%

\subsection{Properties of our Nonlocal Operators} \label{ss:properties_nonlocal_op}
We are now ready to state our main assumptions for the operator $\mathcal{L}$, defined by the expression \eqref{e:nonlocal_operator} together with Dirichlet \eqref{e:nonlocal_D} or Neumann \eqref{e:nonlocal_N} boundary conditions.
We state these in terms of the convolution kernel $\gamma$. 

%\JMS{I think these conditions should be somehow connected to \ref{Q1}-\ref{Q3}, even if simply to say these are additional assumptions. LC: I think this is now done below (in red by Gabriela).}
\begin{assumption}\label{a:gamma}
We assume that $\gamma(z)$ is a radial function in $L^1(\R^n)$, with average one and finite second moment.
In addition, we suppose that $\gamma$ can be written as
\[ \gamma \ = \ \gamma_e + \gamma_s,\]
with 
\[ \Gamma_e  \ = \ \int_{\R^n} \gamma_e(z) \;dz, \qquad  \Gamma_s \ = \ \int_{\R^n} \gamma_s(z) \;dz,  \]
 and where
\begin{itemize}
\item $\gamma_e(z) $ generates a nonlocal operator of diffusive type, i.e., $\mathcal{L}_e: L^2(\R^n) \longrightarrow L^2(\R^n)$,
\[ \mathcal{L}_e u \ = \ \Gamma_e \int_{\R^n} (u(y) -u(x)) \tilde{\gamma}_e(|x-y|) \;dy, \qquad  \tilde{\gamma}_e = \gamma_e/ \Gamma_e \]
whose Fourier symbol satisfies assumptions \ref{a:operatorL} and \ref{a:decomposeL}; as a consequence, $\gamma_e$, and so $\gamma_s$, is in $L^1(\R^n)$.

\item $\gamma_s(z) $ is small in the sense that it is either strictly negative, or 
 $$\|\gamma_s\|_{L^1(\R)} - \Gamma_s \ < \ \beta \ = \ \min\{ \beta_0,\beta_1\},$$
 where $\beta_k>0$ is the largest number such that
\[ \beta_k \| u\|^2_{X^k} \ \leq \ (-\mathcal{L}_e u, u)_{L^2(\Omega)} \ = \ B_e[u,u],\]
for all $u \in X^k$. In this case,  $B_e : X \times X \longrightarrow \R$ is associated bilinear form and the space $X= X^k_{D/N}$  accounts for the corresponding boundary conditions and the necessary regularity $k=0,1$. 

\item When considering nonlocal Neumann boundary constraints, we further assume that $\gamma_e(z) $ is strictly positive.
\end{itemize}

\end{assumption}

The above assumptions on $\gamma$ guarantee that  bilinear forms, $B_{D/N}$, given as in Definitions \ref{d:bilinearform} and \ref{d:bilinearform_N},
are bounded and coercive in a space of functions with sufficient regularity. This allows us to satisfy the compactness conditions, \ref{Q1}-\ref{Q3}, of the Mountain-Pass Theorem.  
Indeed, when $\gamma_s =0$, the boundedness and coercivity follow from  Propositions \ref{p:bilinearform1} and \ref{p:bilinear_N}.
To extend the result to the case when $\gamma_s \neq 0$, notice that because $\gamma_s \in L^1(\R)$, we may conclude that these bilinear forms are bounded. If $\gamma_s$ is not strictly negative, the coercivity is a consequence of first writing
\begin{align*}
 \mathcal{L} u \ = & \  \gamma \ast u - u \\
\mathcal{L} u \ = & \ \Gamma_e ( \tilde{\gamma}_e \ast u - u) + \gamma_s \ast u - \Gamma_s u,
 \end{align*}
and the next set of inequalities:
\begin{align*}
\beta \| u\|^2_{L^2(\Omega)} \ \leq & \ B_e[u,u]\\
\beta \| u\|^2_{L^2(\Omega)}  \ \leq & \ \underbrace{ B_e[u,u]  -  (\gamma_s \ast u, u)_{L^2(\Omega)} +\Gamma_s (u,u)_{L^2(\Omega)} }_{=B_{D/N}[u,u]}\\
& + (\gamma_s \ast u, u)_{L^2(\Omega)} - \Gamma_s (u,u)_{L^2(\Omega)}\\
\beta  \| u\|^2_{L^2(\Omega)}  \ \leq & \ B_{D/N} [u,u] +(  \| \gamma_s\|_{L^1(\R)} - \Gamma_s )  \| u\|^2_{L^2(\Omega)}  \\
(\beta + \Gamma_s - \| \gamma_s\|_{L^1(\R)} )   \| u\|^2_{L^2(\Omega)}  \ \leq & \ B_{D/N} [u,u].
\end{align*}
The assumptions on $\gamma_s$ then allows us to conclude that $\beta + \Gamma_s - \| \gamma_s\|_{L^1(\R)}  >0$.
A similar argument also gives us
\[ (\beta +\Gamma_s  - \| \gamma_s\|_{L^1(\R)} )   \| u\|^2_{H^1 (\Omega)}   \leq  B_{D/N}[u,u],\]
from which we conclude that $B_{D/N}: X_{D/N} \times X_{D/N} \longrightarrow \R$ is coercive. 

If on the other hand the kernel $\gamma_s(x)<0$ for all $x$, we then can write
\begin{align*}
\mathcal{L} u \ & = \ \gamma \ast u - u \\
 & = \Gamma_e(  \tilde{\gamma}_e \ast u - u)  -  |\Gamma_s|( \tilde{\gamma}_s \ast u - u)
 \end{align*}
where as before $\tilde{\gamma}_e  = \gamma_e / \Gamma_e $ and $\tilde{\gamma}_s =  \gamma_s/ \Gamma_s$.
Since $\tilde{\gamma}_s$ is now strictly positive, the results from \cite[Lemmas 2.3 and 2.5]{cappanera2024analysis} show that the bilinear form, $B_s: X^0_{D/N} \times X^0_{D/N} \longrightarrow \R$ defined by
\[ B_s[u,u] = |\Gamma_s| ( \tilde{\gamma}_s \ast u - u, u)_{L^2(\Omega)} \]
is coercive. As a result,
\begin{align*}
\beta_e \| u\|^2_{X} & \ \leq  \ B_e[u,u]\\
\beta_e \| u\|^2_{X}  & \ \leq \ \underbrace{ B_e[u,u] +B_s[u,u] }_{=B_{D/N}[u,u]} - B_s[u,u]\\
\beta_e \| u\|^2_{X}  \ \leq \ & B_{D/N}[u,u] - \beta_s \| u\|^2_{L^2(\Omega)}\\
\beta_e \| u\|^2_{X}  \ \leq \ & B_{D/N}[u,u] 
  \end{align*}
where $X = L^2(\Omega)$ or $ H^1(\Omega)$. We summarize these result in the next proposition.

\begin{proposition}\label{p:bilinear_general}
Let $\gamma \in L^1(\R^n)$ satisfy Assumption  \ref{a:gamma} and consider the bilinear form  $B_{D/N}: X^1_{D/N} \times X^1_{D/N} \longrightarrow \R $ defined by
\[ B_{D/N}[u,\phi] \ = \ \frac{1}{2} \iint_{\R^n\times \R^n }  (u(y) - u(x)) \gamma(|x-y|) (\phi(y) -\phi(x)) \;dy\;dx.\]
Then there exist positive constants $C_i$ and $\beta_i$, $i =1,2,$ such that
\begin{enumerate}[label=\textbf{T\arabic*}]
\item\label{T1} $ B_{D/N}[u,\phi] \leq C_1 \| u\|_{L^2(\Omega)} \| \phi\|_{L^2(\Omega)}$
\item\label{T2} $ B_{D/N}[u,u] \geq \beta_1 \| u\|^2_{L^2(\Omega)}$
\item\label{T3} $ B_{D/N}[u,\phi] \leq C_2 \| u\|_{H^1(\Omega)} \| \phi\|_{H^1(\Omega)}$
\item\label{T4} $ B_{D/N}[u,u] \geq \beta_2 \| u\|^2_{H^1(\Omega)}$
\end{enumerate}
for all $u, \phi \in X^1_{D/N}$.
\end{proposition}

\begin{remark}
In general it is difficult to find the value of the coercivity constant, $\beta$, associated with a bilinear form.
In the case of the nonlocal operators studied in Section \ref{Subsec: PropL} this constant corresponds to the smallest eigenvalue of the operator, which can be approximated as follows. The Fourier Transform, $L_e(\xi)$, of the operator corresponds to the curve of the continuous spectrum of $\mathcal{L}_e:L^2(\R^n) \longrightarrow L^2(\R^n)$. When we restrict our problem to a bounded domain, the eigenvalues of the operator will land on this curve. The smallest eigenvalue would then be close to the origin, $ \xi =0$, where by assumption the Fourier symbol has an expansion
\[ L_e(\xi) = -\alpha |\xi|^2 + \mathrm{O}(|\xi|^4)\]
with $\alpha $ being the second moment of the kernel $\gamma_e$. It then follows that
\[ \beta \sim \alpha \frac{1}{d(\Omega)^2} ,\]
where $d(\Omega)$ represents the diameter of the domain.
\end{remark}

%%%%%%%%%%%%%%%%%%

\subsection{Kernel Examples}\label{Subsec: kernelEx}

Now we discuss examples of convolution kernels $\gamma$ satisfying Assumption \ref{a:gamma}.

\noindent{\bf Example 1:} Logistic kernel
\[ \gamma(x ) = \frac{1}{\Gamma} (1 + (|x| /a)^b)^{-1 }, \qquad \Gamma = \int_{\R^n} (1 + (|x| /a)^b)^{-1 } \;d x .\]
Here, $a$ controls the width of the kernel, while $b$ needs to satisfy $b > n +2$ in order for the kernel to have a finite second moment.

\begin{figure}[t] %  figure placement: here, top, bottom, or page
   \centering
    \includegraphics[width=3in]{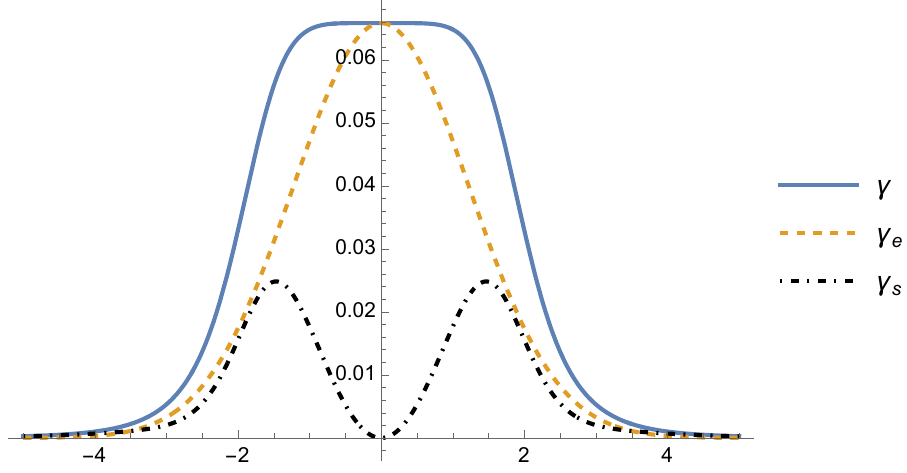} 
   \caption{Plot of convolution kernels $\gamma(x) = \frac{1}{\Gamma} (1 + (|x| /a)^b)^{-1 }, \gamma_e(x) = \frac{1}{\Gamma} \exp( -|x|^2/c^2),$ and $\gamma_s(x) = \gamma(x)- \gamma_e(x)$, with $n =2, a = 2, p=6, c=3$, corresponding to Example 1.  Here the horizontal axis represents $|x|$.}
   \label{fig:Example1}
\end{figure}

In the case when $n=2$, $a = 2$, and $b =6$  we can split the kernel as $\gamma = \gamma_e+ \gamma_s$ with
\[ \gamma_e (x) \ = \ \frac{1}{\Gamma} \exp( -|x|^2/c), \qquad \text{with} \quad c = 3, \quad \Gamma \ = \ \frac{2 \pi^2 a^2}{3\sqrt{3}}. \]
The diffusive operator associated with $\gamma_e$ is then given by the expression
\[ \mathcal{L}_e u(x) \ = \ \Gamma_e \int_{\R^2} (u(y) -u(x) ) \tilde{\gamma}_e(|x-y|) \;dy, \qquad  \tilde{\gamma_e} = \gamma_e/ \Gamma_e \]
where 
\[ \Gamma_e \ = \ \int_{\R^2} \gamma_e(|x|)\;dx \ = \ \frac{\pi c^2}{\Gamma}.\]
When defined on $L^2(\R^2)$, this operator then has a Taylor expansion,
\[ \hat{\mathcal{L}}_e(\xi) \ = \ \Gamma_e \left( - c^2 |\xi|^2 + \mathrm{O}(|\xi|^4) \right),  \]
so that the coercivity constant $\beta \sim \Gamma_e c^2 \frac{1}{d(\Omega)^2} =  \frac{3 \sqrt{3}}{2\pi} \frac{c^4}{a^2} \frac{1}{d(\Omega)^2} = \frac{2.422}{d(\Omega)^2}$.

At the same time we have that $\gamma_s = \gamma- \gamma_e$. The choice $c = 3$ was made so that  $\gamma_s$ is strictly positive (see Figure \ref{fig:Example1}), so we obtain $\Gamma_s = \| \gamma_s\|_{L^1(\R^2)}$.
Thus, $\beta + \Gamma_s - \|\gamma_s\|_{L^1(\R^2)} >0$.

\noindent{\bf Example 2:} Power Law kernel
\[ \gamma(x)\  = \ \frac{1}{\Gamma} (1+ |x|/a)^{-p}, \qquad \Gamma \ = \ \int_{\R^n} (1+ |x|/a)^{-p} \;dx . \]
Again we have that $a$ controls the width of the kernel, while $p$ needs to satisfy $p > n +2$ in order for the kernel to have a finite second moment.

\begin{figure}[t] %  figure placement: here, top, bottom, or page
   \centering
   \includegraphics[width=3in]{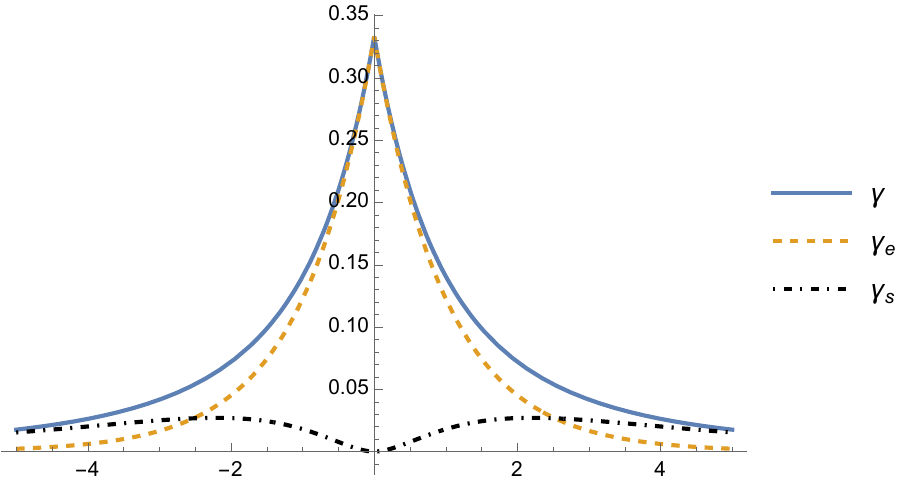} 
   \caption{Plot of convolution kernels $\gamma(x) = \frac{1}{\Gamma} (1+ |x|/a)^{-p}, \gamma_e(x) = \frac{1}{\Gamma} \exp(-|x|/c),$ and $\gamma_s(x) = \gamma(x)- \gamma_e(x)$, with $n =1, a = 3, p=3, c=1$, corresponding to Example 2.}
   \label{fig:Example2}
\end{figure}

For the case when $n=1$, $a = 3$, and $p =3$, we can write $\gamma= \gamma_e + \gamma_s$ with
\[ \gamma_e (x) \ = \ \frac{1}{\Gamma} \exp(-|x|/c), \qquad c \ = \ 1, \quad \Gamma \ = \ \frac{2a}{p-1}.\]
The diffusive operator associated with $\gamma_e$ is then given by the expression
\[ \mathcal{L}_e u \ = \ \Gamma_e \int_{\R} (u(y) -u(x) ) \tilde{\gamma}_e(|x-y|) \;dy, \qquad  \tilde{\gamma_e} \ = \ \gamma_e/ \Gamma_e \]
where 
\[ \Gamma_e \ = \ \int_{\R} \gamma_e(|x|)\;dx \ = \ \frac{2c}{\Gamma}.\]
When defined over $L^2(\R)$, this operator then has a Fourier series expansion,
\[ \hat{\mathcal{L}}_e(\xi) \ = \ \Gamma_e \left( -c^2\xi^2 + \mathrm{O}(|\xi|^4) \right) \]
so that the coercivity constant $\beta \sim \Gamma_e c^2 \frac{1}{d(\Omega)^2} = \frac{p-1}{2a} \frac{c^3}{d(\Omega)^2} $.

Now, the choice of $c =1$ guarantees that $\gamma_s = \gamma - \gamma_e $ is strictly positive (see Figure \ref{fig:Example2} ), so that $\Gamma_s = \| \gamma_s\|_{L^1(\R)}$ and thus $\beta + \Gamma_s - \|\gamma_s\|_{L^1(\R)} >0$.

%%%%%%%%%%%%%%%
\noindent{\bf Example 3: } Inverted Mexican Hat kernel

\[ \gamma(x) \ = \ \frac{1}{\Gamma} (A \exp(-x^2/a^2) -B \exp(-x^2/b^2)) , \qquad \Gamma \ = \ \int_{\R^n} (A \exp(-x^2/a^2) -B \exp(-x^2/b^2))   \;dx .\]
To get an inverted Mexican Hat kernel we assume $0<a<b$, and pick $A,B>0$ such that $(Aa^2-Bb^2)<0$. 

For concreteness, we consider here the case when $n=1, a=1, b=2, A = 2, B=1$, giving us $\Gamma = \sqrt{\pi}(Aa^2-Bb^2)<0$ and split $\gamma = \gamma_e + \gamma_s$ with
\[\gamma_e(x) = - \frac{B}{\Gamma} \exp(-x^2/b^2) \qquad \Gamma_s = \int_\R \gamma_s(x) = -\sqrt{\pi} b^2 \frac{B}{\Gamma}>0 \]
\[\gamma_s(x) = \frac{A}{\Gamma} \exp(-x^2/a^2) \qquad \Gamma_s = \int_\R \gamma_s(x)\;dx = \sqrt{\pi} a^2 \frac{A}{\Gamma}<0 \]
Since $\gamma_s(x)$ is strictly negative (see also Figure \ref{fig:Example3} ), we may apply Proposition \ref{p:bilinear_general}. 

\begin{figure}[t] %  figure placement: here, top, bottom, or page
   \centering
   \includegraphics[width=3in]{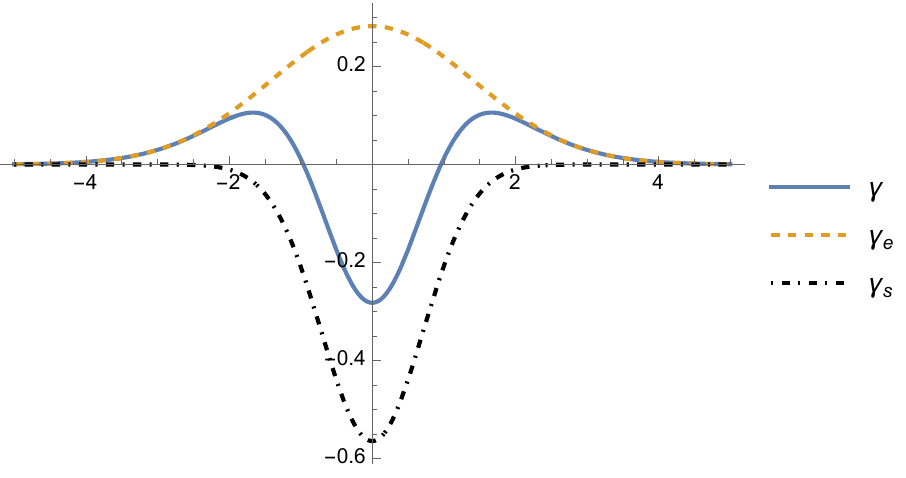} 
   \caption{Plot of convolution kernels $\gamma(x) = \frac{1}{\Gamma}  (A \exp(-x^2/a^2) -B \exp(-x^2/b^2))  , \gamma_e(x) = - \frac{B}{\Gamma} \exp( -|x|^2/b^2),$ and $\gamma_s(x) = - \frac{A}{\Gamma} \exp( -|x|^2/a^2)$, with $n =2, a = 2, p=6, c=3$, corresponding to Example 1. Here the horizontal axis represents $|x|$.}
   \label{fig:Example3}
\end{figure}

 Kernels listed in Examples 1 and 2 are functions used in applications where they describe probability density functions modeling the spread of plant seeds, see \cite{bullock2017}. Example 3 is a variation of the Mexican Hat kernel used in neural field models, where it represents how neurons can excite or inhibit their neighbors, see for instance \cite{kang2003}. The `standard' Mexican Hat kernel models long-range inhibition and short range excitation \cite{ermentrout2015}, but in contrast the kernel used in Example 3, it does not satisfy Assumption \ref{a:gamma}. 

%%%%%%%
\subsection{Properties of the Nonlinearity}\label{ss:nonlinearity}

Just as in \cite{siktar2024superlinear}, we need a series of assumptions on the nonlinearity $f$; the last of these conditions, \ref{A4}, is chosen in favor of the standard Ambrosetti-Rabinowitz condition, since we do not want to be limited to functions $f$ which are odd in the variable $u$. This condition appears with $\mu = 2$ in \cite{torres2016existence}.

\begin{assumption}[Assumptions on data]\label{Assump: AdmissibleControls}
     Let $a_1, a_2 > 0$, $\mu > 2$, $r \in (n + 1, \infty)$ be fixed constants, and let $\al \in (1, 2^*/2)$ be a fixed exponent where $2^* := \frac{2n}{n - 2}$ when $n \geq 3$ and $2^* := \infty$ when $n \leq 2$. 
     %Also let $C: (0, \infty) \rightarrow (0, \infty)$ be fixed. 
     We consider right-hand side datum $f: \Om \times \R \rightarrow \R$ satisfying the following conditions:
    \begin{enumerate}[label=\textbf{A\arabic*}]
    % \item \label{A1} We have the inequality
    %  $\|f\|_{W^{1, r}(\Om \times (-M, M))} \ \leq \ C(M)$ for all $M > 0$. 
        \item \label{A2} We have that 
        \begin{equation}\label{Eq: fGrowth}
            |f(x, t)| \ \leq \ a_1 + a_2|t|^{\al}
        \end{equation}
        for all $x \in \Om$ and $t \in \R$.
        \item\label{A3} We have the limit
        \begin{equation}\label{Eq: uniformLimit}
            \lim_{t \rightarrow 0}\frac{f(x, t)}{t} \ = \ 0
        \end{equation}
        uniformly over $x \in \Om$.
        \item\label{A4} There exists $\th \geq 1$ so that the following holds independently of $f$: if we define 
        \begin{equation}\label{Eq: Gantideriv}
            F(x, t) \ := \ \int^{t}_{0}f(x, \tau)d\tau
        \end{equation}
        then $\th\mathcal{F}(x, t) \geq \mathcal{F}(x, \sig t)$ for $(x, t) \in \R^n \times \R$ and $\sig \in [0, 1]$, where we have defined
        \begin{equation}\label{Eq: ScrFDef}
            \mathcal{F}(x, t) \ := \ tf(x, t) - \mu F(x, t).
        \end{equation}
        \item\label{A5} We have the limit
        \begin{equation}\label{Eq: uniformLimitInfinity}
            \lim_{|t| \rightarrow \infty}\frac{f(x, t)}{t} \ = \ \pm \infty
        \end{equation}
        uniformly over $x \in \Om$, and at least one of the limits $t \rightarrow \pm \infty$ is such that \eqref{Eq: uniformLimitInfinity} equals $+\infty$.
    \end{enumerate}
\end{assumption}

Note that the restriction on $\al$ guarantees that the compact embedding $H^1(\Om) \subset \subset L^{2\al}(\Om)$  \cite[Theorem 9.16]{brezis2011functional}.
As a result, $F: \Om \times \R \rightarrow \R$ satisfies the following:
        \begin{enumerate}[label=\textbf{H\arabic*}]
        \item\label{H1} $F$ has a first-order continuous derivative and first-order continuous derivative with respect to the second argument.
        \item We have (by the definition \eqref{Eq: Gantideriv} and \eqref{Eq: fGrowth})
        \begin{equation}\label{Eq: FGrowth}
        |F(x, t)| \ \leq\ a_1|t| + \frac{a_2}{\al + 1}|t|^{\al + 1}
        \end{equation}
        for a.e. $x \in \Om$ and all $ t \in \R$.
    \end{enumerate}

%%%%%%%%%%%%%%%%%%%%%

%\section{Analysis of the state equation}\label{Sec: StateEqnAnalysis}

\section{Existence of Nontrivial Solutions via Mountain-Pass Theorem}\label{s:mountain-pass}

Our goal in this section is to prove existence of non-trivial solutions to 
\begin{equation}\label{e:problem2}
-\mathcal{L}u =f(x,u) \qquad x \in \Omega \subset \R^n,
\end{equation}
with either nonlocal Dirichlet,
\[ u = 0 \quad x \in \Omega^c,\]
or nonlocal Neumann,
\[\mathcal{L} u = 0 \quad x \in \Omega^c,\]
boundary constraints.
 Our strategy is to use the  Mountain-Pass Theorem, which requires us to define an energy functional associated to our problem. Here we let $I_{D/N}: X \longrightarrow \R$ be given by
\begin{equation}\label{Eq: convolveEnergy}
 I_{D/N}[u] \ = \ \frac{1}{2} B_{D/N}[u,u] - \int_\Omega F(x,u) \;dx,
 \end{equation}
where $F(x,u) = \int_0^u f(x,v) \;dv$ and the space $X$ is either $X^1_D$ in the case of Dirichlet boundary constraints, or $X^1_N$ in the case of Neumann boundary constraints. 

We now state the Mountain-Pass Theorem (originally proven in \cite{rabinowitz1982mountain}), which gives conditions that guarantee existence of non-trivial critical points of an energy functional $I \in C^1(X,\R)$. Here $C^1$ refers to the existence of a \textit{Fréchet derivative} of $I[\cdot]$, and $X^*$ refers to the dual space of a Banach Space $(X, \|\cdot\|_X)$.

\begin{theorem}[Mountain-Pass Theorem]\label{t:mountain-pass} 
Let $(X, \|\cdot\|_X)$ be a Banach space and take $I \in C^1(X; \R)$ to be a functional. Suppose that
\begin{enumerate}
\item[\bf G1] There is a constant $\rho \in \R_+$ such that

\[ \inf_{\|u\|_X=\rho} I[u] \ > \ I[0]=0 .\] 
\item[\bf G2] There is another constant $\be \leq 0$ and an element $e \in X$ with norm $\| e\|_{X} > \rho$ such that $I[e] \leq \be $.
\item[\bf P-S] Every sequence $\{u_n\}$ in $X$ with $\{ I[u_n]\}$ bounded in $\R$ and satisfying $I'[u_n] \to 0$ in $X^*$ has a convergent sub-sequence.
\end{enumerate}
Then, the value $c$, defined as
\begin{equation}\label{e:critical_value}
 c \ := \ \inf_{\gamma \in \Gamma} \max_{t \in [0,1]} I[\gamma(t)]
 \end{equation}
where
\[ \Gamma \ = \ \{ \gamma \in C([0,1], X): \gamma(0) = 0 \quad \gamma(1) = e\}\]
is a critical value of $I:X \longrightarrow \R$. That is, there is an element $w \in X$ such that
\[ I[w] = c, \quad \text{and} \quad I'[w] =0.\]
\end{theorem}
%\JMS{Conditions G1 and G2 essentially say the same thing. We should replace G1 with: there exist $\rho, \be > 0$ such that for all $u \in X$ with $\|u\|_X = \rho$, we have $I[u] > \be$, Lemma \ref{Lem: theRingLemma}. Agreed?} \LC{condition G2 is not mentioned in the rest of the text so I would just remove it.} \GJ{I think both conditions are necessary. G1 says that at radius r, the energy is positive. G2 says that there is a 'point' with norm greater than r, where the energy is negative (or less than $\beta$.}

%

In the forthcoming subsections we prove that the energy $I[\cdot]$ satisfies the geometric conditions {\bf G1} and {\bf G2} and the compactness assumption {\bf P-S} stated in Theorem \ref{t:mountain-pass}. As a result we will obtain our main theorem, which we state here.

\begin{theorem}[Multiplicity of solutions to state equation]\label{Thm: Multiplicity}
  Consider the nonlocal equation \eqref{e:nonlocal_eq} and take $X\subset H^1(\R^n)$ to be the appropriate Banach space encoding  boundary constraints. That is, $X = X^1_D$ in the case of Dirichlet, while $X = X^1_N$ in the case of Neumann boundary constraints.
Then, if $f $ satisfies Assumptions \ref{Assump: AdmissibleControls},  the nonlocal problem has at least one non-trivial weak solution $u \in X$.
\end{theorem}

As a first step, it is straightforward to show that
critical points of these energies correspond to weak solutions of problem \eqref{e:problem2}.

\begin{lemma}\label{l:bilinear}
Suppose $u  \in X$ is a critical point for $I[u]$. Then  $u$ is a weak solution of \eqref{e:problem2}.

\end{lemma}
\begin{proof}
%Suppose $u \in \widetilde{H^1}(\Omega)$ satisfies \eqref{e:minimum}. 
Fix any $v \in X$ and write
\[ i(\tau) \ = \ I [u + \tau v] \quad \tau \in \R.\]
%Because $u+ \tau v$ belongs to the admissible set $X$, then the function $i(\tau)$ has a minimum at $\tau = 0$. 
Since $u$ is a critical point, we have $\frac{d i}{d\tau} (0) = 0$. 
At the same time, with the notation $u = u(x)$ and $u' = u(y)$, we have that
\begin{equation*}
\begin{aligned}
i(\tau) & =  \frac{1}{4} \iint_{\R^n \times \R^n} \gamma(|x-y|) (u+ \tau v - u' + \tau v')^2 \;dy \;dx - \int_\Omega F(x,u+ \tau v) \;dx\\
\frac{d i}{d\tau} (\tau) & =  \frac{1}{2}  \iint_{\R^n \times \R^n} \gamma(|x-y|) (u+ \tau v - u' + \tau v')(v-v') \;dy \;dx - \int_\Omega f(x,u+ \tau v) v \;dx\\
\frac{d i}{d\tau} (0) & =  \frac{1}{2}  \iint_{\R^n \times \R^n} \gamma(|x-y|) (u - u')(v-v') \;dy \;dx - \int_\Omega f(x,u) v \;dx\\
\frac{d i}{d\tau} (0) & =  ( -\mathcal{L} u, v)_{L^2(\Omega)} - ( f(\cdot,u) , v)_{L^2(\Omega)}
\end{aligned}
\end{equation*}
Since this last condition holds for all $v  \in X$ it follows that $u\in X$  is a weak solution to
\[ -\mathcal{L} u \ = \ f(x,u) \quad x \in \Omega.\]

\end{proof}

%%%%%%%%%%%%%

\subsection{Geometric Conditions}

To prove the first required geometric condition, we need Lemma 3 in \cite{servadei2012mountain}; it applies in our setting  since $f$ satisfies Assumption \ref{Assump: AdmissibleControls} (namely \eqref{Eq: fGrowth} and \eqref{Eq: uniformLimit}).
\begin{proposition}\label{Prop: fB}
    Let $f: \Om \times \R \rightarrow \R$ be a Carathéodory function satisfying \ref{A2} and \ref{A3}, then for any $\ep > 0$ there exists $\de(\ep) > 0$ so that for a.e. $x \in \Om$ and all $t \in \R$,
    \begin{equation}\label{Eq: fB}
        |f(x, t)| \ \leq \ 2\ep|t| + (\al + 1)\de(\ep)|t|^{\al}
    \end{equation}
        and
    \begin{equation}\label{Eq: FB}
        |F(x, t)| \ \leq \ \ep|t|^2 + \de(\ep)|t|^{\al + 1}.
    \end{equation}
    \end{proposition}

Now, the first geometric condition ${\bf G1}$ to be proven is a variant of \cite[Proposition 9]{servadei2012mountain}.
\begin{lemma}\label{Lem: theRingLemma}
    There exist constants $\rho, \be > 0$ such that, for all $u \in X$ with 
    $\| u\|_{H^1(\Omega)} = \rho$,
    %$\|\grad u\|_{L^2(\R^n)} = \rho$, 
    we have $I[u] > \be$.
  %  \JMS{Notice that $\|\grad \cdot\|_{L^2(\R^n)}$ is a norm on $X^1_D$ and is arguably the most convenient one to use here.}
\end{lemma}

\begin{proof}
    Let $u \in X$ be arbitrary for now. As a first step, from Lemma \ref{l:bilinear}
\begin{equation}\label{Eq: theRingLemma1}
    I[u] \ \geq \ \frac{\beta_2}{2}\| u\|^2_{H^1(\Omega)} - \int_{\Om}F(x, u(x))dx.
\end{equation}
Next, fix $\ep > 0$ and use \eqref{Eq: FB} to obtain
\begin{equation}\label{Eq: theRingLemma2}
    I[u] \ \geq \ \frac{\beta_2}{2}\| u\|^2_{H^1(\Omega)} - \int_{\Om}\ep|u(x)|^2 + \de(\ep)|u(x)|^{\al + 1}dx.
\end{equation}
By H{\"o}lder's inequality,
\begin{equation}\label{Eq: theRingLemma3}
    I[u] \ \geq \ \frac{\beta_2}{2}\| u\|^2_{H^1(\Omega)} - \ep|\Om|^{\frac{2\al - 2}{\al + 1}}\|u\|^2_{L^{\al + 1}(\Om)} - \de(\ep)\|u\|^{\al + 1}_{L^{\al + 1}(\Om)}.
\end{equation}
Now due to the embedding $H^1(\Om) \subset L^{2\al}(\Om) \subset L^{\al + 1}(\Om)$ we have
\begin{equation}\label{Eq: theRingLemma4}
    I[u] \geq \ \frac{\beta_2}{2}\| u\|^2_{H^1(\Omega)} - \ep|\Om|^{\frac{2\al - 2}{\al + 1}}M^2_{\al + 1}\|u\|^2_{H^1(\Om)} - \de(\ep)M^{\al + 1}_{\al + 1}\|u\|^{\al + 1}_{H^1(\Om)}.
\end{equation}
Factor to obtain
\begin{equation}\label{Eq: theRingLemma6}
    I[u] \ \geq \ \| u\|^2_{H^1(\Omega}\left(\frac{\beta_2}{2}- \ep|\Om|^{\frac{2\al - 2}{\al + 1}}M^2_{\al + 1} - \de(\ep)M^{\al + 1}_{\al + 1}(C_P + 1)^{\al - 1}\|\grad u\|^{\al - 1}_{L^2(\R^n)}\right).
\end{equation}
Pick $\ep > 0$ sufficiently small so that $\frac{\beta_2}{2}- \ep|\Om|^{\frac{2\al - 2}{\al + 1}}M^2_{\al + 1}\geq \frac{1}{4}$. Then we may pick $u$ so that $\| u\|_{H^1(\Omega)} = \rho$ for $\rho$ small enough so that the lower bound in \eqref{Eq: theRingLemma6} equals some positive value $\be$ (for instance, so the quantity in parentheses exceeds $1/8$), and that completes the proof.
\end{proof}

This next lemma, which verifies \textbf{G2}, is a variant of \cite[Proposition 10]{servadei2012mountain}. We cannot use \cite[Lemma 4]{servadei2012mountain} because $f$ and $F$ do not satisfy an Ambrosetti-Rabinowitz condition; the technique shown here provides a workaround.

\begin{lemma}\label{Lem: coercivityLemH^1}
Let $\rho > 0$ be as defined in Lemma \ref{Lem: theRingLemma}.
There exists $e \in X$ with norm $\| e\|_{H^1(\Omega)} > \rho$ such that $I[e] < 0$. 
\end{lemma}

\begin{proof} We follow the technique in the proof of \cite[Theorem 1.1]{torres2016existence}. Assume without loss of generality that in assumption \ref{A5}, the limit $\lim_{t \rightarrow \infty}\frac{f(x, t)}{t} = +\infty$; otherwise use a non-positive test function in the steps that follow. As a result of this limit,
\begin{equation}\label{Eq: coercivityLemH^1Eq1}
    \lim_{t \rightarrow \infty}\frac{F(x, t)}{t^2} \ = \ +\infty.
\end{equation}
Thus, for any  $\ep\in(0,1)$ there exists $K > 0$ such that for any $t > K$, we have that $F(x, t) > \frac{t^2}{\ep}$. To determine a lower bound for $F$ on $ \R^n \times \R^+$, we now fix $(x,t)\in \R^n \times[0,K]$. Using the condition \eqref{Eq: FGrowth}, we get:
\begin{equation}\label{Eq: FDecay}
F(x,t) \ \geq \ -a_1 K - \frac{a_2}{\alpha+1} K^{\alpha+1} \ \geq \ -a_1 K - \frac{a_2}{\alpha+1} K^{\alpha+1} -K^2 \ =: \ -\tilde{K}.
\end{equation}
It then follows that for all $(x, t) \in \R^n \times \R^+$ we have 

\begin{equation}\label{Eq: FGetsLarge}
F(x, t) \ \geq \ \frac{t^2}{\ep} - \frac{\tilde{K}}{\ep}.
\end{equation}

Next, letting $u \in X$ with $\|u\|_{H^1(\Omega)} = 1$ and $u \geq 0$ for a.e. $ x\in \R^n$, we obtain
%Let $u \in C^{\infty}_0(\R^n)$ be such that $u \geq 0$ a.e. in $\R^n$ and then by the previous assertions we have that
\begin{equation}\label{Eq: coercivityLemH^1Eq2}
    \int_{\Omega}\frac{F(x, tu(x))}{t^2}dx \ \geq \ \frac{1}{\ep}\|u\|^2_{L^2(\Omega)} - \frac{\tilde{K}}{\ep t^2}|\Omega|.
\end{equation}

Consequently, we have that
\begin{equation}\label{Eq: coercivityLemH^1Eq4}
\begin{aligned}
    \lim_{t \rightarrow \infty}\frac{I[tu]}{t^2} \ & =  \  \lim_{t \rightarrow \infty}\left( \frac{1}{2t^2} B[tu,tu]
    %\frac{1}{2}\|\grad u\|^2_{L^2(\R^n)} + \frac{1}{2}\|\Lin * \grad u\|^2_{L^2(\R^n)}
     - \int_{\Om}\frac{F(x, tu(x))}{t^2}dx\right) \\
     & \ \leq \ \lim_{t \rightarrow \infty} \left( \left(\frac{C_1}{2}- \frac{1}{\ep}\right)\|u\|^2_{L^2(\Omega)} + \frac{\tilde{K}}{\ep t^2}|\Omega| \right) \\
     & \ \leq \ - \infty,
     \end{aligned}
\end{equation}
where the last line follows from picking $\ep$ sufficiently small. To complete the proof, we may pick $t > \rho >0$ sufficiently large and set $e := tu$. 
 \end{proof}
 
 Notice that if we further assume that the function $F(x,t)$ satisfies
  \begin{equation}\label{e:F_even}
    \lim_{|t| \rightarrow \infty}\frac{F(x, t)}{t^2} \ = \ +\infty
    \end{equation}
 uniformly for a.e. $x \in \Omega$, we obtain a stronger result, which we state next.  
 %to the analysis that forms the basis of the numerical scheme for numerically compute nontrivial critical points of $I[\cdot]$.
 
 \begin{lemma}\label{Lem: coercivityLemH^1_v2}
Let $\rho > 0$ be as defined in Lemma \ref{Lem: theRingLemma} and 
assume $F(x,t)$ satisfies \eqref{e:F_even}. Then, there is a constant $\rho_2>\rho$
such that for all $e \in X$ with norm $\| e\|_{H^1(\Omega)} > \rho_2$ we have that $I[e] \leq 0$.
\end{lemma}

The proof  follows just as in Lemma \ref{Lem: theRingLemma}, except we do not require our test function $u$ to be nonnegative. We will use Lemma \ref{Lem: coercivityLemH^1_v2} in Section \ref{Sec: Special}.

%%%%%%%%%%%%

\subsection{Compactness Conditions}

Now we start working towards proving the Palais-Smale compactness condition \textbf{PS}. To this end we define $\mathscr{F}: X \rightarrow \R$ via 
\begin{equation}\label{Eq: ARFakeMu}
    \mathscr{F}[u] \ := 
    \ \frac{1}{\mu}\int_{\Om}\mathcal{F}(x, u(x))dx \ = \ \int_{\Om}\frac{1}{\mu}f(x, u(x))u(x) - F(x, u(x))dx .
\end{equation}

We need to show this quantity is controlled from below. This comes in three steps, starting with a lower bound on \eqref{Eq: ARFakeMu} for arguments of small norm.

\begin{lemma}\label{Lem: FakeARBoundLem1}
    There exists a $C > 0$ such that for $u \in X$ with $\|u\|_{H^1(\R^n)} \leq 1$, we have the bound
    \begin{equation}\label{Eq: FakeARBoundLem1Eq}
        \mathscr{F}[u] \ \geq \ -C.
    \end{equation}
\end{lemma}

\begin{proof}
We bound both terms of \eqref{Eq: ARFakeMu} separately. To control the term involving $f$, we use \eqref{Eq: fGrowth}, Hölder's inequality, the continuous embedding of $H^1(\Om) \subset L^{\al + 1}(\Om)$, and Poincaré's inequality to obtain
\begin{eqnarray}\label{Eq: FakeARBoundLem1Eq1}
    \begin{aligned}
       & \int_{\Om}f(x, u(x))u(x)dx \ \geq \ \\
       &-\int_{\Om}a_1|u(x)| + a_2|u(x)|^{\al + 1}dx \ = \ \\
       &-a_1\|u\|_{L^1(\Om)} - a_2\|u\|^{\al + 1}_{L^{\al + 1}(\Om)} \ \geq \ \\
       &-a_1|\Om|^{\frac{1}{2}}\|u\|_{L^2(\Om)} - a_2\|u\|^{\al + 1}_{L^{\al + 1}(\Om)} \ \geq \ \\
       &-a_1|\Om|^{\frac{1}{2}}\|u\|_{H^1(\R^n)} - a_2M^{\al + 1}_{\al + 1}\|u\|^{\al + 1}_{H^1(\R^n)}.
    \end{aligned}
\end{eqnarray}
Then since $\|u\|_{H^1(\Om)} \leq 1$ we obtain
\begin{equation}\label{Eq: FakeARBoundLem1Eq2}
    \int_{\Om}f(x, u(x))u(x)dx \ \geq \ -a_1|\Om|^{\frac{1}{2}} - a_2M^{\al + 1}_{\al + 1}.
\end{equation}
Similarly, but using \eqref{Eq: FGrowth} instead of \eqref{Eq: fGrowth}, we obtain
\begin{equation}\label{Eq: FakeARBoundLem1Eq3}
    \int_{\Om}F(x, u(x))dx \ \geq \ -a_1|\Om|^{\frac{1}{2}} - \frac{a_2M^{\al + 1}_{\al + 1}}{\al + 1}.
\end{equation}
Combining \eqref{Eq: FakeARBoundLem1Eq2} with \eqref{Eq: FakeARBoundLem1Eq3}, we set
$$C := \frac{1}{\mu}(a_1|\Om|^{\frac{1}{2}} + a_2M^{\al + 1}_{\al + 1}) + a_1|\Om|^{\frac{1}{2}} + \frac{a_2M_{\al + 1}}{\al + 1}$$
which completes the proof.
\end{proof}

The second step is a scaling argument.

\begin{lemma}\label{Lem: FakeARBoundLem2}
   Let $t \in (1, \infty)$, and let $\th > 1$ be the scaling parameter introduced in \ref{A4}. Then we have
    \begin{equation}\label{Eq: FakeARBoundLem2Eq}
        \mathscr{F}[tu] \ \geq \ \frac{1}{\th}\mathscr{F}[u].
    \end{equation}
\end{lemma}

\begin{proof}
We use the definition of $\mathcal{F}$ from \ref{A4} and its associated scaling inequality:
\begin{eqnarray}\label{Eq: FakeARBoundLem2Eq1}
    \begin{aligned}
        \mathscr{F}[tu] \ &= \ \frac{1}{\mu}\int_{\Om}\mathcal{F}(x, tu(x))dx \\
        \ &\geq \ \frac{1}{\mu \th}\int_{\Om}\mathcal{F}(x, u(x))dx \\
        \ &= \ \frac{1}{\th}\mathscr{F}[u],
    \end{aligned}
\end{eqnarray}
completing the proof.
\end{proof}

Here is the final step for getting the bound we need on $\mathscr{F}[\cdot]$.

\begin{lemma}\label{Lem: FakeARBoundLem3}
     Let $C > 0$ be the constant appearing in Lemma \ref{Lem: FakeARBoundLem1} and let $\th > 1$ be the scaling parameter introduced in \ref{A4}. Then we have
     \begin{equation}\label{Eq: FakeARBoundLem3Eq}
         \mathscr{F}[u] \ \geq \ -\frac{C}{\th},
     \end{equation}
     for any $u \in X$.
\end{lemma}

\begin{proof}
If $\|u\|_{H^1(\Om)} \leq 1$ then the proof is complete by Lemma \ref{Lem: FakeARBoundLem1}. Otherwise let $w := \frac{u}{\|u\|_{H^1(\Om)}}$ and we may apply Lemma \ref{Lem: FakeARBoundLem2} with $t := \|u\|_{H^1(\Om)}$ followed by Lemma \ref{Lem: FakeARBoundLem1} to conclude
\begin{equation}\label{Eq: FakeARBoundLem3Eq1}
    \mathscr{F}[u] \ = \ \mathscr{F}[tw] \ \geq \ \frac{1}{\th}\mathscr{F}[w] \ \geq \ -\frac{C}{\th},
\end{equation}
completing the proof.
\end{proof}

Since proving \textbf{PS} requires a compactness argument, we must show that sequences approximating critical points are bounded.

\begin{lemma}\label{Lem: Deriv->0}
    Let $\{u_j\}^{\infty}_{j = 1} \subset X$ be a sequence such that \begin{equation}\label{Eq: PSCriterionDeriv}
    \lim_{j \rightarrow \infty}I'[u_j] \ = \ 0
    \end{equation}
    and 
    \begin{equation}\label{Eq: PSCriterion}
    \lim_{j \rightarrow \infty}I[u_j]\ = \ c
    \end{equation}
    for some $c \in \R$. Then $\{u_j\}^{\infty}_{j = 1}$ is bounded in $X$.
\end{lemma}

\begin{proof}
    %Due to the convergence properties of $I_f$ and $I_f'(\cdot)$, 
    Let $v_j = \frac{u_j}{\| u\|_{H^1(\Omega)}}$.
    Due to our assumptions, there exists $\ka > 0$ so that for all $j \in \N^+$ sufficiently large we have that
    \begin{equation}\label{Eq: Deriv->0Eq1}
        %\left|I_f'(u_j)\frac{u_j}{\|\grad u_j\|_{L^2(\R^n)}}\right|
         \left| \langle I'[u_j], v_j \rangle  \right| \ \leq \ \ka\mu.
    \end{equation}
    and
    \begin{equation}\label{Eq: Deriv->0Eq2}
        |I[u_j]| \ \leq \ \ka.
    \end{equation}
    It then follows that, for all large $j \in \N^+$, 
\begin{equation}\label{Eq: PSUB2}
   I[u_j] - \frac{1}{\mu}I'[u_j]u_j \ \leq \ \ka(1 + \| u_j\|_{H^1(\Omega)}).
\end{equation}
where the constant $\mu$ is the same as in Assumption \ref{A4}.

Now, to find a lower bound for the difference described in \eqref{Eq: PSUB2}, notice that for any $j \in \N^+$ \eqref{e:nonlocal_operator} and Proposition \ref{l:bilinear} (or Proposition \ref{p:bilinear_N}) give us
    \begin{eqnarray}\label{Eq: PS1}
        \begin{aligned}
            I[u_j] - \frac{1}{\mu}I'[u_j]u_j \ = \ & \left( \frac{1}{2} - \frac{1}{\mu} \right) B[u_j,u_j] 
            + \int_{\Om}\frac{1}{\mu}f(u_j(x))u_j(x) - F(x, u_j(x))dx, \\
           \ \geq \ & \left( \frac{1}{2} - \frac{1}{\mu} \right) \beta_2 \| u_j\|^2_{H^1(\Omega)} - \frac{C}{\theta},
        \end{aligned}
    \end{eqnarray}
   where the inequality follows from Proposition \ref{l:bilinear} (or Proposition \ref{p:bilinear_N}) and Lemma \ref{Lem: FakeARBoundLem3},

Finally, combining \eqref{Eq: PS1} with \eqref{Eq: PSUB2} we obtain,
\begin{equation}
    \left(\frac{1}{2} - \frac{1}{\mu}\right)\be_2\| u_j\|^2_{H^1(\Omega)} - \frac{C}{\th} \ \leq \ \ka(1 + \| u_j\|_{H^1(\Omega)}),
\end{equation}
from which the boundedness of $\{u_j\}^{\infty}_{j = 1}$ in $X$ follows, since $\mu > 2$.
\end{proof}

\begin{lemma}[Palais-Smale Compactness]\label{Lem: PalaisSmale}
    Let $\{u_j\}^{\infty}_{j = 1} \subset X$ be a bounded sequence in $X$ such that $I'[u_j] \rightarrow 0$ as $j \rightarrow \infty$. Then there exists $u \in X$ such that, up to a not relabeled sub-sequence, $u_j \rightarrow u$ strongly in $X$.
\end{lemma}

\begin{proof} We proceed similarly to \cite[Proposition 12]{servadei2012mountain}. First, by reflexivity, there exists a $u \in X$ so that $u_j \rightharpoonup u$ weakly in $X$ (up to a not-relabeled sub-sequence). Then we claim for all $v \in X$,
\begin{equation}\label{Eq: WeakConvergenceB}
\lim_{j \to \infty} B[u_j, v] \ = \ B[u,v].
\end{equation}
More specifically, since $u_j \rightharpoonup u$ weakly in $X$, and $B[\cdot, \cdot]$ is symmetric and satisfies the identity \eqref{Eq:BilinearToL_D} in the Nonlocal Dirichlet case (or \eqref{Eq:BilinearToL_N} in the Nonlocal Neumann case), we have
\begin{equation}\label{Eq: ClosednessAdmissibleEq1A}
             \lim_{j \rightarrow \infty}B[u_j, v] \ = \ \lim_{j \rightarrow \infty}B[v, u_j] \ = \ \lim_{j \rightarrow \infty}\lang -\Lin v, u_j\rang_{L^2(\Om)} \ = \ \\
             \lang -\Lin v, u\rang_{L^2(\Om)} \ = \ B[v, u] \ = \ B[u, v] .
\end{equation}

Due to the weak convergence and compact embeddings, there exists a further, not relabeled sub-sequence so $u_j \rightarrow u$ strongly in $L^{2\al}(\Om)$ and a.e. in $\Om$. In addition, due to \cite[Theorem 9]{brezis1983theorie}, there exists a function $\ell \in L^{2\al}(\Om)$ so that $|u_j(x)| \leq \ell(x)$ a.e. $x \in \Om$, for all $j \in \N^+$. With this in mind, we claim that
\begin{equation}\label{Eq: PSConvEq3}
    \lim_{j \rightarrow \infty}\int_{\Om}f(x, u_j(x))u_j(x) \ = \ \int_{\Om}f(x, u(x))u(x)dx.
\end{equation}
First notice that the integrands exhibit pointwise convergence due to the continuity of $f$. Thus to invoke the Dominated Convergence Theorem, we must show that the integrands have a pointwise, $L^1(\Om)$ upper bound uniformly in $j$. We use \eqref{Eq: fGrowth} to choose a candidate for such a bound:
\begin{eqnarray}\label{Eq: PSConvEq4}
\begin{aligned}
    &f(x, u_j(x))u_j(x) \ \leq \ \ell(x)(a_1 + a_2\ell(x)^{\al}) \ = \ a_1\ell(x) + a_2\ell(x)^{\al + 1}
    \end{aligned}
\end{eqnarray}
The upper bound is finite since $\ell \in L^{2\al}(\Om)$. So, we may apply the Dominated Convergence Theorem to obtain \eqref{Eq: PSConvEq3}. By the same argument, we have that for any $k \in \N^+$,
\begin{equation}\label{Eq: PSConvEq6}
    \lim_{j \rightarrow \infty}\int_{\Om}f(x, u_j(x))u_k(x)dx \ = \ \int_{\Om}f(x, u(x))u_k(x)dx.
\end{equation}
Our assumption $I'[u_j] \to 0$, together with the previous limits, shows that
\[\lim_{j \to \infty} \frac{1}{2} B[u_j,u_k] \ = \ \lim_{j \to \infty} \int_\Omega f(x, u_j(x))u_k(x)dx = \int_\Omega f(x,u(x))u_k(x) \;dx \]
We can therefore write
\[ \int_\Omega f(x,u(x))u(x) \;dx = \lim_{j \to \infty} \int_\Omega f(x, u_j(x))u_j(x)dx \ = \ \lim_{j \to \infty} \frac{1}{2} B[u_j,u_j]  \]
where the last equality follows from picking a diagonal subsequence, which we do not relabel.
Because the bilinear form $B: X \times X \longrightarrow \R$ defines an equivalent norm in the space $X$, (see Proposition \ref{p:bilinear_N}), it follows that $u_j \to u$ strongly in $X$.

\end{proof}

Having proven \textbf{PS} finishes the proof of Theorem \ref{Thm: Multiplicity}.

%%%%%%%%%
\section{Existence of Nontrivial Solutions for Special Geometries}\label{Sec: Special}
In this section we provide an alternative proof for the existence of non-trivial weak solutions to our nonlocal problem \eqref{e:nonlocal_eq}. Our results are summarized in Theorem \ref{t:existence_2} and form the basis for the numerical scheme presented in Section \ref{s:numerics}. The analysis given here extends the results from \cite{bailova2021mountain} to nonlocal equations with nonlinearities satisfying Assumption \ref{Assump: AdmissibleControls} as well as the extra condition \eqref{e:F_even}. This additional restriction on $f$ guarantees that the energy functional attains a negative value for all points on a ball of sufficiently large radius. We will see in Section \ref{s:numerics} that the resulting geometric structure makes it possible to find critical points using gradient descent.

As in the previous section, we consider the energy, $I: X \longrightarrow \R$, associated with equation \eqref{e:problem2} and defined by
\begin{equation}\label{e:energy_repeat}
 I[u] \ = \ \frac{1}{2} B[u,u] - \int_\Omega F(x,u) \;dx.
 \end{equation}
Again, we take $X = X^1_D$ in the case of Dirichlet and $X = X^1_N$ in the case of Neumann boundary constraints.
We then have that  the corresponding bilinear forms $B:X \times X \longrightarrow \R$, defined in Subsection  \ref{ss:properties_nonlocal_op}, are both bounded and coercive. Recall as well that the nonlinearity $f(x,u)$, appearing in our nonlocal problem \eqref{e:nonlocal_eq}, satisfies Assumption \ref{Assump: AdmissibleControls} and defines $F(x,u) = \int_0^u f(x,v)\;dv$.

\begin{theorem}\label{t:existence_2}
Let $I: X \longrightarrow \R$, given in \eqref{e:energy_repeat}, be the energy associated with the nonlocal problem \eqref{e:nonlocal_eq}. If we further assume that the nonlinear function $F(x,t)$ satisfies
\begin{equation}\label{a:F}
 \lim_{|t| \to \infty} \frac{F(x,t)}{t^2} \ = \ \infty
 \end{equation}
uniformly for a.e. $ x\in \Omega$, then
\[ \bar{c} \ = \ \inf_{w \in X, w\neq 0} \max_{t\geq 0} I[tw]\]
is a critical value of $I[\cdot]$.
\end{theorem}

%%%

Our goal is to show that there exists an element $\bar{u} \in X$ such that $I[\bar{u}] = \bar{c}$ and
\[ \bar{c} \ = \ \inf_{w \in X, w\neq 0} \max_{t\geq 0} I[tw] \ = \ c \]
where $c$ denotes the critical value of the energy $I[\cdot]$ given in Theorem \ref{t:mountain-pass} equation \eqref{e:critical_value}.
To do this, we need the next lemma.

%%%

\begin{lemma}\label{l:min_max}
Let $I: X \longrightarrow \R$, satisfy the assumptions of Theorem \ref{t:existence_2}. Then,
\[ \bar{c} \ = \ \inf_{w \in X, w\neq 0} \max_{t\geq 0} I[tw] \ = \ \min_{w \in X, w\neq 0} \max_{t\geq 0} I[tw].  \]
\end{lemma}

\begin{proof}
Let $w \in X$ with $\|w \|_{H^1(\Omega)} =1$, and define
\[ g_w(t) \ = \ I[tw] \ = \ t^2 B[w,w] - \int_\Omega F(x,tw) \;dx.\]
Notice that $g: \R_+ \to \R$ is a smooth function of $t$. 
We want to show that $g$ attains its maximum on $\R_+$.

First, Proposition \ref{Prop: fB} implies that $g_w(0) =0$. Next, Lemma \ref{Lem: theRingLemma} shows there are numbers $\rho,\beta >0$ such that for $t= \rho$ the function $g_w(t) >\beta$, independently of $w$.
Since $F(x,t)$ satisfies property \ref{a:F}, from Lemma \ref{Lem: coercivityLemH^1_v2} we have that there exists a constant $\rho_2> \rho$ such that for all $t\geq \rho_2$ the function $g_w(t) <0$, independently of $w$. It follows that $g_w(t)$ attains its maximum for some value $t^* \in (0, \rho_2)$.

Let $G(u) = g_w(t^*)$, with $u := t^* w$ and notice that
\[ \bar{c} \ = \ \inf_{\substack{w \in X,\\ w\neq 0}} \max_{t\geq 0} I(tw) \ = \ \inf_{u \in K } G(u)\]
where $K$ is the closed ball of radius $\rho_2$ in $X$, i.e.,
\[ K \ = \ \{ u \in X: \| u\|_{H^1(\Omega)} \leq \rho_2\}.\]

{
From the compact Sobolev embedding $H^1(\Omega) \subset \subset L^p(\Omega)$, valid for any $p \in (2, 2^*)$ with $2^* =  2n/(n-2)$, we see in addition that $K$ is a compact subset of $L^p(\Omega)$. Since $F$ satisfies Assumption \ref{Assump: AdmissibleControls} and picking $ p>\alpha +1>2$, it follows that $G: L^p(\Omega) \longrightarrow \R$ is a bounded functional. Indeed, we can write
\[ |G(u)| \ \leq \ |B[u,u] | +  \int_\Omega \left| F(x,u) \right| \;dx\]
and bound first
\begin{align*}
B[u,u] \ & = \ (- \mathcal{L} u, u)_{L^2(\Omega)} \\
| B[u,u] | & \ \leq \ \|\mathcal{L}u \|_{L^2(\Om)} \| u\|_{L^2(\Omega)} \\
| B[u,u] | & \ \leq \ \|\mathcal{L}\| \|u \|^2_{L^2(\Om)} \\
| B[u,u] | & \ \leq \ C \|\mathcal{L}\| \|u \|^2_{L^p(\Om)},
\end{align*}
where the last inequality follows from the embedding $L^p(\Omega) \subset L^r(\Omega)$, valid for bounded domains $\Omega$ when $p>r$, and the fact that the operator $\mathcal{L}:L^2(\Omega) \longrightarrow L^2(\Omega)$ is bounded.

Next, from Proposition \ref{Prop: fB}, we have that for any $\ep >0$ there is a $\delta(\ep) >0$ such that 
\[ \left| F(x,u) \right| \ \leq \ \ep |u|^2 + \delta(\ep) |u|^{\alpha+1}.\]
As a result,
\begin{align*}
\int_\Omega \left| F(x,u) \right| \;dx & \ \leq \ \int_\Omega \ep |u|^2 + \delta(\ep) |u|^{\alpha+1} \;dx\\
\int_\Omega \left| F(x,u) \right| \;dx & \ \leq \ \ep \|u\|_{L^2(\Omega)}^2 + \delta(\ep)  \|u\|_{L^{\alpha+1}(\Omega)}^{\alpha+1} \\
\int_\Omega \left| F(x,u) \right| \;dx & \ \leq \ C( \ep \|u\|_{L^p(\Omega)}^2 + \delta(\ep)  \|u\|_{L^{p}(\Omega)}^{\alpha+1} ) ,
\end{align*}
where again the last inequality follows from selecting $p> \alpha +1$ and using the embedding $L^p(\Omega) \subset L^{\alpha+1}(\Omega)$.

We therefore conclude that the functional $G:L^p(\Omega) \longrightarrow \R$  attains its infimum in $K$, and there is a $\bar{u} \in K$ such that 
\[  \inf_{u \in K} G(u) \ = \ \bar{c} \ = \ G(\bar{u}) \ = \  \min_{u \in K } G(u).\] 
}

%From the compact Sobolev embedding $H^1(\Omega) \subset \subset L^p(\Omega)$ with $2< p< 2^* = 2n/(n-2)$, we see in addition that $K$ is a compact subset of $L^p(\Omega)$. 

%Since $F$ satisfies Assumption \ref{Assump: AdmissibleControls} and $ 2<p<2^*$, it follows that $G: L^p(\Omega) \longrightarrow \R$ is a bounded functional. Consequently it attains its infimum in $K$ and there is a $\bar{u} \in K$ such that 
%\[  \inf_{u \in K} G(u) = \bar{c}  = G(\bar{u})  =  \min_{u \in K } G(u).\] 
\end{proof}

\begin{proof}[Proof Theorem \ref{t:existence_2} ]
The results from Lemma \ref{l:min_max} show there is an element $\bar{u} \in X$ such that 
\begin{equation}\label{Eq: IDNInitial}
I[ \bar{u}] \ = \ \bar{c} \ = \  \inf_{\substack{w \in X, \\ w\neq 0}} \max_{t\geq 0} I[tw].
\end{equation}
We now want to show that 
\begin{equation}\label{e:c=bar_c}
 I[ \bar{u}] \ = \ c \ = \ \inf_{\gamma \in \Gamma} \max_{t \in [0,1]} I[\gamma(t)],
 \end{equation}
where the set
\[ \Gamma \ = \ \{ \gamma \in C([0,1], X): \gamma(0) = 0 \quad \gamma(1) = e\},\]
with the element $e \in X$ satisfying
\[ \sigma \ := \ I[e] \ < \ I[0] \ = \ 0.\]
%Notice that by the Mountain-Pass theorem $c$ is a critical value of $I[\cdot]$, and that without loss of generality we may assume that $\sigma <0$. %Recall from the Mountain-Pass Theorem that $c$ is a critical value of our energy. 

To prove \eqref{e:c=bar_c}, we introduce again the trajectories
$g_w(t) =  I[tw]$,
with $w \in X $ satisfying $\|w\|_{H^1(\Omega)} = 1$. Let $T_w>0$ be such that
\[g_w(T_w) \ = \ I[T_w w] \ = \ I[e] \ = \ \sigma.\]
Construct a continuous path, $\tilde{\gamma}(t)$, joining the trajectory $g_w(t)$ for $t\in (0,T_w)$ and a second curve which stays on the level set $\{I[w]= \sigma\}$ and connects $g_w(T_w)$ to the element $e$. Rescale the $t$-variable so that $\tilde{\gamma}(t) \in \Gamma$. It then follows that
\begin{equation}\label{Eq: cIneq}
c \ = \ \inf_{\gamma \in \Gamma} \max_{t \in [0,1]} I[\gamma(t)] \leq \inf_{\substack{w \in X,\\ \|w\| =1}} \max_{t\geq 0}g_w(t)= \bar{c}.
\end{equation}

To obtain the reverse inequality of \eqref{Eq: cIneq}, observe that, from Lemma \ref{Lem: coercivityLemH^1_v2}, given any $w \in X$ with $\|w \|_{H^1(\Omega)} =1$,
 we have that
\[ \max_{t\geq 0} g_w(t)  \ = \ \max_{t\in [0,T_w]} g_w(t).\]
It follows that the set of maximum values
\[ M \ = \ \Big\{ \max_{t\in [0,T_w]} g_w(t) : w \in X, \quad \text{with} \quad \|w \|_{H^1(\Omega)} =1\Big\}\]
contains all possible local maxima of $I[\cdot]$. In other words, $M$ contains the set
\[ m \ = \ \Big\{  \max_{t \in [0,1]} I[\gamma(t)] : \gamma \in \Gamma\Big\}.\]
As a result,
\[ \bar{c} \ = \ \inf_{\substack{w \in X, \\ w\neq 0}} \max_{t \geq 0} g_w(t)\  \leq \ \inf_{\gamma \in \Gamma} \max_{t \in [0,1]} I[\gamma(t)] = c.\]
and we obtain the conclusion of the theorem.

\end{proof}

\section{Numerical scheme and illustrations} \label{s:numerics}
In this section, we propose an algorithm to solve the problem \eqref{e:nonlocal_eq} based on the algorithm proposed in \cite{bailova2021mountain} for problems with p-Laplacian and homogeneous Dirichlet boundary conditions.

%%%%%%%%
\subsection{Algorithm}\label{s:numerics_algo}
In the following, we describe the various steps of our algorithm for approximating a solution of the problem \eqref{e:nonlocal_eq} supplemented by homogeneous nonlocal Dirichlet or nonlocal Neumann boundary conditions. For clarity, we drop the subscript $D/N$ when referring to the space $X_{D/N}^1$ or the functional $I_{D/N}$. 

For any function $u \in X$, we denote by $t^*_u$ a positive real number that satisfies
$$I[t^*_u u]  = \max_{t\geq 0} I[tu].$$
The existence of $t^*_u$ is given by Lemma \ref{l:min_max}. We also denote the operator norm on the energy space of $I$ as $\|\cdot\| := \|\cdot\|_{\Lin(X, \R)}$.

The algorithm for approximating a non-trivial solution $u$ of problem \eqref{e:nonlocal_eq} reads as follows.
\begin{itemize}
    \item[0.] Let $\epsilon>0$ be a given tolerance and $\delta>0$ be a given step length.
    
    \item[1.] Let $u_1$ be an arbitrary nonzero function in $X$ that represents an initial guess.
    
    \item[2.] Find $t^*_{u_1}$ and set $w_1 = t^*_{u_1} u_1$.
    
    \item[3.] If $\|I^\prime [w_1] \| \leq \epsilon$, return $w_1$. Go to step 4 otherwise.
    
    \item[4.] Find $v_1 \in X$ such that
    \begin{equation}\label{e:algo_step4}
        I^\prime[w_1] v_1 < 0.
    \end{equation}

    \item[5.] Set $\tilde{w_1} = w_1 + \delta v_1$.
    
    \item[6.] While $I[t^*_{\tilde{w_1}} \tilde{w_1}] \geq I[w_1]$, set $v_1=v_1/2$ and return to Step 5. Go to Step 7 otherwise.
    
    \item[7.] Set $w_1=t^*_{\tilde{w_1}} \tilde{w_1}$ and return to Step 3.
\end{itemize}

The following subsection provides more information on how to proceed through Steps 3 and 4  which require us to evaluate $\|I^\prime[w_1]\|$ and find $v_1\in X$ satisfying \eqref{e:algo_step4}.

\begin{remark}\label{remark:algo_step2}
In Step 2, the value of $t^*_u$ depends on the choice of nonlinearity $f(x,u)$ so no general formula is provided. In Section \ref{s:numerics_results}, we consider different functions $f$ to perform numerical tests and give an explicit formula for $t^*_u$ in each case.
\end{remark}

\begin{remark}\label{remark:algo_step6}
In Step 6, we note that the value of $I(t^*_{\tilde{w_1}} \tilde{w_1})$ is guaranteed to decrease below $I(w_1)$ as we decrease the factor $\delta$ multiplying $v_1$. This is a consequence of the definition of $I^\prime$. Indeed, we have
$$\lim_{\delta\rightarrow 0} \frac{I[w_1 + \delta v_1] - I[w_1]}{\delta} = I^\prime[w_1] v_1 < 0.$$
Therefore, for $\delta>0$ small enough we get $I[w_1 + \delta v_1] - I[w_1] < 0$.
\end{remark}

%%%%%%%%
\subsection{Implementation Details}

We focus here on the Dirichlet nonlocal problem. The results presented here can be extended to the Neumann problem by replacing $X_D^1$ with $X_N^1$ and $I_D$ with $I_N$. 

Let $w_1\in X_D^1$ such that $I_D^\prime[w_1]\neq 0$. Consider the problem: find $b \in X_D^1$ such that the following holds for all $v\in X_D^1$
\begin{equation}\label{e:num_pb_b}
%10^{-10} b    \int_\Omega \nabla b \cdot \nabla v = I_D^\prime[w_1]v .
    B_D[b,v] \ = \ I_D^\prime[w_1]v,
\end{equation}
where $I^\prime_D[w_1] \in \mathcal{L}(X_D^1,\mathbb{R})$ is defined by
$$I^\prime_D[w_1]v 
\ = \  B_D[w_1,v] - \int_\Omega f(x,w_1)v \;dx.$$

\begin{lemma}\label{lemma:num_algo_details}
Let $b$ be the solution of the problem \ref{e:num_pb_b} with $w_1\in X_D^1$ such that $I^\prime[w_1]\neq 0$.
Then, we have
%$$C_P \| b\|_{X_D^1} \leq \|I_D^\prime[w_1]\| \leq \| b\|_{X_D^1},$$
\begin{equation}\label{lemma:num_bound_Iprime}
\beta_2 \| b\|_{H^1(\Omega)} \ \leq \ \|I_D^\prime[w_1]\| \ \leq \ C_2 \| b\|_{H^1(\Omega)},
\end{equation}
and
$$ I_D^\prime[w_1] \left(\frac{-b}{\|b\|_{X_D}}\right) \ < \ 0,$$
where the constants $C_2$ and $\beta_2$ are introduced in Proposition \ref{p:bilinear_general}.
\end{lemma}

\begin{proof}
We recall that $X_D^1$ is equipped with the $H^1$ norm. Therefore, we have
$$\|I_D^\prime[w_1]\| \ = \ \sup_{\substack{v \in X_D^1,\\ \|v\|_{H^1(\Omega)} =1}} |I_D^\prime[w_1] v|.$$
Let $v\in X_D^1$ be such that $\|v\|_{H^1(\Omega)}=1$. Using the property \ref{T3}, we get
$$ |I_D^\prime[w_1] v|
\ = \ \left| B[b,v] \right|
\ \leq \ C_2 \| b \|_{H^1(\Omega)} \| v\|_{H^1(\Omega)}
\ = \ C_2 \| b \|_{H^1(\Omega)}
,$$
which yields the upper bound of \ref{lemma:num_bound_Iprime}.

The lower bound results from applying the property \ref{T4} with $v=\frac{b}{\|b\|_{H^1(\Omega)}}$. It reads
$$ \| I_D^\prime[w_1] \| 
\ \geq \ \left| I_D^\prime[w_1] \frac{b}{\|b\|_{H^1(\Omega)}} \right| 
\ = \ \left| B\left[ b, \frac{b}{\|b\|_{H^1(\Omega)}}\right] \right|
\ = \ \frac{| B[b,b] |}{\| b\|_{H^1(\Omega)} }
\ \geq \ \beta_2 \| b\|_{H^1(\Omega)}.
$$
To prove the second inequality, we write
$$I_D^\prime[w_1] \left( \frac{-b}{\|b\|_{H^1(\Omega)}}\right)
%\ = \ - B[ b, \frac{b}{\|b\|_{H^1(\Omega)}}]
\ = \ - \frac{ B[b, b] }{\|b\|_{H^1(\Omega)}} \leq 0.
$$
We note that the inequality is strict as $B[b,b]$ being zero implies that $b$, and so $I_D^\prime[w_1]$, are zero leading to a contradiction.
\end{proof}

Note that Lemma \ref{lemma:num_algo_details} first allows us to get an upper bound for $\|I_D^\prime[w_1]\|$ which we use to check if $\|I_D^\prime[w_1]\| \leq \epsilon$ in step 3 of the above algorithm. For numerical validation, see following section, we will use the following stopping criterion in step 3
$$\|b\|_{H^1(\Omega)} \leq \epsilon,$$
where we dropped the constant $C_2$ from the upper bound \ref{lemma:num_bound_Iprime}.
Moreover, the above Lemma also provides us with a definition for $v_1$ in step 4 of the algorithm by setting
$$v_1 = \frac{-b}{\|b\|_{H^1(\Omega)}},$$
with $b$ being the weak solution of the problem \ref{e:num_pb_b}.

%%%%%%%%
\subsection{Numerical illustrations}\label{s:numerics_results}
We validate the algorithm described in Section \ref{s:numerics_algo} by considering various setups in which we impose homogeneous nonlocal Dirichlet or nonlocal Neumann boundary conditions. We consider three types of kernels that are either algebraically or exponentially decaying. The solution $b$ of the problem \ref{e:num_pb_b} is approximated using a second order finite element method where the basis functions are defined as piecewise linear Lagrangian polynomials.

% In all the tests reported in this section, we set the domain $\Omega$ to $(-\pi,\pi)$, the tolerance $\epsilon$ to $10^{-3}$ and the step $\delta$ of the gradient descent algorithm, used in step 5, to one. Finally, the initial guess is always set to $u_1(x)=\sin(x)$.

For a given mesh size $h > 0$, used to approximate the solution of \eqref{e:num_pb_b} with piecewise linear finite elements, we denote by $u_h^*$ the approximation of the solution of problem \eqref{e:nonlocal_eq} generated by our algorithm. As exact solutions are not known, we confirm the convergence of our algorithm by computing the norms $L^1$ and $L^2$ of the residual of \eqref{e:nonlocal_eq} defined by:
$$
\begin{array}{lll}
  R_{L^1}   & \ := \ & \int_\Omega |-\mathcal{L} u_h^*(x) - f(x,u_h^*)| dx ,\\
  R_{L^2}   & \ := \ & \left( \int_\Omega (-\mathcal{L} u_h^*(x) - f(x,u_h^*))^2 dx \right)^{\frac{1}{2}} .
\end{array}
$$
To further validate our algorithm, we use piecewise linear finite elements to approximate the solution $\bar{u}$ of the problem
\begin{equation}\label{e:linear_pb_cvg}
-\mathcal{L} \bar{u}  \ = \   f(x,u^*)  \quad  x \in  \Omega,
\end{equation}
supplemented with the same nonlocal boundary conditions used to obtain $u^*$.
For a given mesh size $h$, we can then compute the errors in $L^1$ and $L^2$ norms between $u_h^*$ and $\bar{u}_h$ given by
$$E_{L^1} \ = \ \left( \int_\Omega |u_h^* - \bar{u}_h | dx \right),$$
$$E_{L^2} \ = \ \left( \int_\Omega (u_h^* - \bar{u}_h )^2 dx \right)^{\frac{1}{2}}.$$

\begin{remark}\label{remark:algo_limitation} One could suggest using a classic Newton iterative solver to obtain an approximation of the solution to the nonlocal problem \ref{e:nonlocal_eq}. However, applying such a strategy to the examples considered in this section was shown to lead to solutions that do not
%approximation that does not 
correctly approximate the equation. In particular, the residuals $R_{L^1}$-$R_{L^2}$ and errors $E_{L^1}$-$E_{L^2}$ are large and do not decrease with the mesh size. Therefore, developing algorithms like the one proposed here is essential to solve the problem \ref{e:nonlocal_eq} and could also be used to provide adequate initial guesses for Newton solvers. We will explore this direction in the future.
\end{remark}

% \begin{remark}
%     We note that the setups considered in this section all involve nonlinearity $f(x,t)$ that are odd in $t$. This restriction is a consequence of hypothesis \ref{A4}. Numerical results confirmed that the proposed algorithm does not converge if the nonlinearity has even part in $t$. On the other hand, one test described in section \ref{sec:num_case_4} shows that not satisfying the hypothesis \ref{A3} can still lead to a convergent algorithm.
% \end{remark}

%%%%
\subsubsection{Case 1. Positive exponentially decaying kernel and $f(x,t) = t^3$.} \label{sec:num_case_1}
We set the domain $\Omega$ to $(-\pi,\pi)$ and supplement the problem \ref{e:nonlocal_eq} with homogeneous nonlocal Dirichlet boundary conditions. The initial guess is $u_1(x)=\sin(x)$.
We note that the function $f(x,t)=t^3$ satisfies the hypotheses \ref{A3} and\ref{A5}. Furthermore, hypothesis \ref{A2} is satisfied for any $a_1>0$ with $a_2=1$ and $\alpha=3$, and hypothesis \ref{A4} is satisfied for any $\mu\in(2,4)$ with $\theta=1$. 
Similarly to the proof of \cite[Lemma 4]{bailova2021mountain}, a solution $t^*_u$ of $I(t^*_u u)  = \max_{t\geq 0} I(tu)$ is given by
$$t^*_u \ = \ \left( \frac{B[u,u]}{\int_\Omega u^4} \right)^{\frac{1}{2}}.$$
We consider an exponentially decaying kernel defined by
$$ \gamma(x) \ = \ \frac{1}{2} exp(-|x|).$$
% $$
% \begin{array}{ccl}
%  \gamma(x) \ = \    &  \gamma_1(x) \ = \ & \frac{1}{2} exp(-|x|),  \\
%  \gamma(x) \ = \   &  \gamma_2(x) \ = \ & \frac{1}{\sqrt{\pi}} exp(-x^2). 
% \end{array}
% $$
We run our algorithm with five different mesh sizes $h\in \{0.314, 0.157, 0.078, 0.039, 0.019\}$. The $L^1$-$L^2$ residuals and errors $E_1$-$E_2$ are shown in Table \ref{tab:num_case_1a}. The table also displays the number of degree of freedom, $n_{dof}$, the number of iterations, $it$, required for our algorithm to converge, and the run time $t_{run}$ in seconds. Our results confirm the correct behavior of the proposed algorithm, as the residuals and errors converge with orders close to one in $L^1$ and one-half in $L^2$.
\begin{table}[H]
    \centering
    \caption{Convergence results for case 1 with $\epsilon=10^{-3}$, $\delta=1$, $u_1(x)=\sin(x)$.}%Running time $t_{run}$ is displayed in seconds.}
    \begin{tabular}{|c|c||c|c||c|c||c|c|}
    \hline
    h & $n_{dof}$ & $R_{L^1}$ & $R_{L^2} $  & $E_{L^1}$ & $E_{L^2}$ & it & $t_{run}$ (s)
    \\ \hline
    %0.628 & 10 & 0.05480320 & 0.03997793 & 0.66911377 & 0.35320342 &  11 & 0.13
    %\\ \hline
    0.314 & 20 & 0.04657529 &0.04744037  & 0.45900719 & 0.26293927 & 14 & 0.40
    \\ \hline
    0.157 & 40 & 0.03240179 & 0.04681579 & 0.27046985 & 0.17679683 & 15  &  1.38
    \\ \hline
    0.078 & 80 & 0.01940995 & 0.03936870 & 0.14647999 & 0.11585964 &  19 & 5.93
    \\ \hline
    0.039 & 160 & 0.01045204 & 0.03033283 & 0.07772337 & 0.07708336 & 24 & 26.4  
    \\ \hline
    0.019 & 320 & 0.00550846 & 0.02237211 & 0.04053661 &0.05238792  & 32 & 111 
    \\ \hline  
    \end{tabular}
    \label{tab:num_case_1a}
\end{table}

% \begin{table}[H]
%     \centering
%     \caption{Convergence results for case 1 with kernel $\gamma_2(x)$, $\epsilon=10^{-3}$, $\delta=1$, $u_1(x)=\sin(x)$.}%Running time $t_{run}$ is displayed in seconds.}
%     \begin{tabular}{|c||c|c||c|c||c|c|}
%     \hline
%     h & $R_{L^1}$ & $R_{L^2} $  & $E_{L^1}$ & $E_{L^2}$ & it & $t_{run}$ (s)
%     \\ \hline
%     0.1 & 0.00284266 & 0.00176542 & 0.74537838 & 0.34696443  & 8 & 0.20
%     \\ \hline
%     0.05 & 0.00069555 & 0.00061655 & 0.74429382 & 0.34673014 & 9 & 0.39
%     \\ \hline
%     0.025 & 0.00036835 & 0.00024027 & 0.74427290 & 0.34667302 & 12 & 1.94  
%     \\ \hline
%     0.0125 & 0.00023057 & 0.00013912 &0.74384627 & 0.34662730 & 13 & 8.22
%     \\ \hline
%     0.00625 &0.00005497 &0.00003815 & 0.74404443 & 0.34664083 & 14 & 36.9  
%     \\ \hline
%     0.003125 &0.00009260 &0.00005973 & 0.74408141 & 0.34664361 & 14 &  141
%     \\ \hline  
%     \end{tabular}
%     \label{tab:num_case_1b}
% \end{table}

%%%%
\subsubsection{Case 2. Inverted Mexican Hat kernel and $f(x,t)=t^3$.} \label{sec:num_case_2}
We now change the kernel to an Inverted Mexican Hat kernel introduced in Section \ref{Subsec: kernelEx}. We set $(a,b,A,B): =(1,1,1,2)$ so the kernel is defined as
$$\gamma(x) = \frac{1}{\pi}\left( exp(-x^2/4) - exp(-x^2) \right).$$
The domain $\Omega$, type of nonlocal boundary conditions, and initial guess $u_1$ remain unchanged.
While the bilinear form $B[u,u]$ now involves the above kernel and is therefore different than in the previous test, the formula for $t^*_u$ remains unchanged.
As in the previous section, we run a series of five tests with mesh sizes $h$ ranging from $0.314$ to $0.019$. The quantities $R_{L^1}, R_{L^2}, E_{L^1}, E_{L^2}$ are shown in Table \ref{tab:num_case_2} and recover a similar convergence rate than in the previous setup (i.e. order one in $L^1$ and half in $L^2$). 
We note that algorithm shows a clear jump in convergence when the mesh size reaches $0.019$. This effect is due to the algorithm converging to a solution close to zero, i.e., the trivial solution. We note that such behavior is expected and depends mostly on the initial guess $u_1$. For instance, if $u_1$ has a very small norm in $X_D^1$, it is guaranteed that the algorithm will converge to the trivial solution. 

%\LC{For this test, we have a super convergence for the finest mesh because we are converging to a constant close to zero. To discuss if we keep it (and add discussion about possibility of converging to trivial solution) or not. I am inclined to keep it and admit that algorithm can converge to trivial solution (depends on initial guess, for instance if initial guess has a very small norm the algorithm will converge to zero immediately).}

\begin{table}[H]
    \centering
    \caption{Convergence results for case 2 with $\epsilon=10^{-3}$, $\delta=1$, $u_1(x)=\sin(x)$.}%Running time $t_{run}$ is displayed in seconds.}
    \begin{tabular}{|c|c||c|c||c|c||c|c|}
    \hline
    h & $n_{dof}$ & $R_{L^1}$ & $R_{L^2} $  & $E_{L^1}$ & $E_{L^2}$ & it & $t_{run}$ (s)
    \\ \hline
    %0.628 & 10 & 0.14500949 & 0.10418775 & 0.49428531 & 0.31372224 & 13 & 0.24
    %\\ \hline
    0.314 & 20 & 0.24209273 & 0.17827050 & 0.44302720 & 0.26543935 & 13 & 0.63
    \\ \hline
    0.157 & 40 & 0.13054000 & 0.13490316 & 0.23372808 & 0.17195871 & 17 & 2.27
    \\ \hline
    0.078 & 80 & 0.07054912 & 0.10298292 & 0.12184988 & 0.11134002 & 24 & 12.7
    \\ \hline
    0.039 & 160 & 0.03505767 &0.07210115 & 0.06295756 &  0.07804388& 35 & 69.7    
    \\ \hline
    0.019 & 320 & 0.00000441 & 0.00000262 & 0.00660828 &0.00263632 & 80 &  572
    \\ \hline  
    \end{tabular}
    \label{tab:num_case_2}
\end{table}

\subsubsection{Case 3. Positive exponentially decaying kernel with $f(x,t)=t^5$.}\label{sec:num_case_3}
For this series of tests, we consider an exponentially decaying kernel
$$\gamma(x) = \frac{1}{\sqrt{\pi}} exp(-x^2). $$
The domain, boundary conditions and initial guess are again unchanged.
The function $f(x,t)$ satisfies the hypotheses \ref{A3}-\ref{A5}. Using  $(a_1,a_2,\alpha)=(1,1,5)$, one can show that hypothesis \ref{A2} is also satisfied. Similarly, hypothesis \ref{A4} holds for $\mu\in(2,6)$ and $\theta=1$.
A solution $t^*_u$ of $I[t^*_u u]  = \max_{t\geq 0} I[tu]$ is given by
$$t^*_u = \left( \frac{B[u,u]}{\int_\Omega u^6} \right)^{\frac{1}{4}}.$$
We perform a series of five tests where the mesh size ranges from $0.314$ to $0.019$ and display our results in Table \ref{tab:num_case_3b}. While the $L^1$-$L^2$ residual shows a small convergence rate, between $[0.5,1]$ in $L^1$ and smaller than one-half in $L^2$, the $L^1$-$L^2$ errors show a convergence rate closer to one. We note that, as in the previous test, a super convergence effect appears for the finest mesh where the algorithm recovers an approximation close to the trivial solution. Moreover, this behavior is prevalent for multiple distinct choices for $\ga$.
\begin{table}[H]
    \centering
    \caption{Convergence results for case 3 with $\epsilon=10^{-3}$, $\delta=1$, $u_1(x)=\sin(x)$.}%Running time $t_{run}$ is displayed in seconds.}
    \begin{tabular}{|c|c||c|c||c|c||c|c|}
    \hline
    h & $n_{dof}$ & $R_{L^1}$ & $R_{L^2} $  & $E_{L^1}$ & $E_{L^2}$ & it & $t_{run}$ (s)
    \\ \hline
    %0.628 & 10 & 0.03012309 & 0.02147429 &  1.35394427 &0.63626496  & 10 & 0.53
    %\\ \hline
    0.314 & 20 & 0.03925108 & 0.03981134 & 1.10925082 &0.52974364  &12 & 0.91
    \\ \hline
    0.157 & 40 & 0.03793965 & 0.05461015 & 0.73062823  &0.35943050 & 17 & 3.04    
    \\ \hline
    0.078 & 80 & 0.02548352 & 0.05187136 & 0.42814897 & 0.22118707 & 17 & 8.51
    \\ \hline
    0.039 & 160 &0.01446269 & 0.04206778 & 0.23404745 & 0.13131264 & 22 & 35.7 
    \\ \hline
    0.019 & 320 & 0.00006810 &0.00003166  & 0.16800524 & 0.06702448 & 56 & 297 
    \\ \hline  
    \end{tabular}
    \label{tab:num_case_3b}
\end{table}

%%%%
\subsubsection{Case 4. Positive exponentially decaying kernel with $f(x,t)=t^3-t$.}\label{sec:num_case_4}
For this test, we consider the same setup as in the previous section, but modify the function $f(x,t)=t^3-t$ so that it no longer satisfies hypothesis \ref{A3}. Indeed, one can show that $\lim_{t \rightarrow 0}\frac{f(x, t)}{t} = -1 \neq 0$. 
% The domain is kept to $\Omega=(-\pi,\pi)$ and the initial guess to $u_1(x)=\sin(x)$. The kernel $\gamma$ is given by
% $$\gamma(x) \ = \ \frac{1}{\sqrt{\pi}} exp(-x^2). $$
We note that the function $f(x,t)$ still satisfies hypothesis \ref{A5} and also \ref{A2} by setting $(a_1,a_2,\alpha)=(1,2,3)$. Moreover, hypothesis \ref{A4} holds for $\mu\in(2,4)$ and $\theta=1$. 
A solution $t^*_u$ of $I[t^*_u u]  = \max_{t\geq 0} I[tu]$ is given by
$$t^*_u \ = \ \left( \frac{B[u,u] + \int_\Omega u^2}{\int_\Omega u^4} \right)^{\frac{1}{2}}.$$

We perform a series of five tests with different mesh sizes and display our results in Table \ref{tab:num_case_4}. We note that while the errors $E_{L^1}-E_{L^2}$ both converge with order one with respect to the mesh size $h$, the residual errors show a rate of convergence closer to half on average.
\begin{table}[H]
    \centering
    \caption{Convergence results for case 4 with $\epsilon=10^{-3}$, $\delta=1$, $u_1(x)=\sin(x)$.}%Running time $t_{run}$ is displayed in seconds.}
    \begin{tabular}{|c|c||c|c||c|c||c|c|}
    \hline
    h & $n_{dof}$ & $R_{L^1}$ & $R_{L^2} $  & $E_{L^1}$ & $E_{L^2}$ & it & $t_{run}$ (s)
    \\ \hline
    %0.628 & 10 & 1.37161091 & 0.74096033 & 0.87050177 &0.38721339& 12 & 0.51
    %\\ \hline
    0.314 & 20 &1.01178678 & 0.74065266 & 0.30090628 & 0.13297895 & 18 & 1.21
    \\ \hline
    0.157 & 40 & 0.60279548 & 0.62296971 & 0.16107430 & 0.07149136 & 26 &  4.00
    \\ \hline
    0.078 & 80 & 0.16438354 & 0.33952388 &0.10686057 &0.04704756 & 84 & 56.6
    \\ \hline
    0.039 & 160 & 0.17161320 & 0.35436100 &0.04594508 & 0.02103514& 44 & 57.4     
    \\ \hline
    0.019 & 320 & 0.08772507 & 0.25600842 &0.02338767 & 0.01110279& 67 & 334 
    \\ \hline  
    \end{tabular}
    \label{tab:num_case_4}
\end{table}

\subsubsection{Case 5. Homogeneous nonlocal Neumann boundary conditions and $f(x,t)=0.5(-u - 3 u^2 + 4 u^3)$.}

We conclude our numerical investigations by considering a problem on the domain $\Omega=(0,3)$ with homogeneous Neumann nonlocal boundary constraints. Compared to the previous example, the nonlinearity is modified once more to $f(x,t)=0.5(-u - 3 u^2 + 4 u^3)$ so it now includes a quadratic term. As a result, neither hypothesis \ref{A3} nor \ref{A4} is satisfied. We note that hypotheses \ref{A2} and \ref{A5} still hold for this choice of $f$. The kernel is set to
$$ \gamma(x) \ = \ \frac{1}{2} exp(-|x|).$$
Unlike our previous setup, we modify the initial guess $u_1$ to a step function defined by
$$
u_1(x) =
\begin{cases}
0 \ \text{ if } \ 0 \leq x < 1, \\
1 \ \text{ if } \ 1 \leq x < 2, \\
0 \ \text{ if } \ 2 \leq x \leq 3 . 
\end{cases}
$$
This setup is motivated by our interest in the generation of steady pulse solutions in Allen-Cahn equations or FitzHugh-Nagumo-type equations where the local diffusion operator $-\Delta u$ is replaced by the operator $-\mathcal{L}u$.

We run a series of five tests with mesh sizes $h$ ranging from $0.3$ to $0.009375$. Our results are displayed in Table \ref{tab:num_case_5} and confirm that the residuals and errors converge in both $L^1$ and $L^2$ norms.
We note that computational times $t_{run}$ are slightly larger than in the previous sections. This is due to the use of Neumann nonlocal boundary constraints that lead us to extend the computational domain to $(-1.5,4.5)$. Therefore, the effective degrees of freedom is increased by a factor of two when compared to the previous setups considered in this section. We refer to \cite{cappanera2024analysis} for more details on the numerical approximation of the problem \ref{e:nonlocal_eq} with nonlocal Neumann boundary constraint.

\begin{table}[H]
    \centering
    \caption{Convergence results for case 5 with $\epsilon=10^{-3}$, $\delta=1$.}%Running time $t_{run}$ is displayed in seconds.}
    \begin{tabular}{|c|c||c|c||c|c||c|c|}
    \hline
    h & $n_{dof}$ & $R_{L^1}$ & $R_{L^2} $  & $E_{L^1}$ & $E_{L^2}$ & it & $t_{run}$ (s)
    \\ \hline
    %0.3 &20 & 0.53650741 & 0.56569346 & 0.48413759 & 0.20363992 & 13 & 0.5
    %\\ \hline
    0.15 & 40 & 0.29477230 & 0.43664263 &0.25616048  & 0.10894473 & 28 & 3.00
    \\ \hline
    0.075 & 80 & 0.15175489 & 0.12040370 & 0.12040370 & 0.12040370 & 53 &  26.5   
    \\ \hline
    0.0375 &160 & 0.03090516 & 0.07208059 & 0.07688191 &  0.03266330 & 55 & 48.7
    \\ \hline
    0.01875 & 320 & 0.01663096 & 0.05488005 & 0.03966971 & 0.01725178 & 75 & 250    
    \\ \hline
    0.009375 & 640 & 0.00863797 & 0.04027057 &0.02010100& 0.00912427 & 108 & 1415 
    \\ \hline  
    \end{tabular}
    \label{tab:num_case_5}
\end{table}

\section{Conclusion}\label{s:conclusion}

This paper proved a multiplicity result for solutions to equations of the form $-\Lin u = f(x, u)$ where $\Lin$ is a convolution operator associated with radially symmetric $L^1$ kernels with finite second moment,
%exponentially decaying kernel,
 and with either a nonlocal Dirichlet or a nonlocal Neumann boundary constraint. We also established numerically that a gradient descent-type algorithm tailored to energy functionals with a mountain pass structure allows for convergence even when standard Newton-type algorithms perform poorly. In the future, we plan to study extensively this algorithm by comparing results obtained with our gradient descent algorithm versus the stationary solution obtained from solving equations, like Allen-Cahn, over long time integration.
 
 %\GJ{We did not do Neumann conditions for the numerical tests, maybe explain here why this is ok}.

%While the motivation for this paper came from various applications, we purposefully did not pursue any of them as test cases for our numerical scheme, leaving this for future work. In this paper we mainly wanted to focus on the proof existence of non-trivial solutions. In addition, this choice comes from the added restrictions on the nonlinearity, which narrow the type of problems that it can handle. For instance, most analysis of population and vegetation models focus on proving existence of invasion fronts. This requires nonlinearities, or reaction terms, that support heteroclinic connections between two fixed points. In the scalar case, the simlest case would be that corresponding ot a Fisehr KKP type of equation, whose nonlinear terms need to be cubic.

Numerous continuations to this project are readily available: for one, one can seek similar multiplicity results for equations with the same underlying operator but different geometric behavior, such as Linking-type energies (see \cite[Chapter 8]{struwe2000variational} for an overview of the Linking Theorem); semilinear problems with the same convolution structure also remain open. In addition, we will develop analytical and numerical results for optimal control problems with a constraint equation of type \eqref{e:nonlocal_operator}, using \cite{siktar2024superlinear} as a guide for the analytic framework.

%%%%%%%%%%%%%%%%%%%%%%%%%%%%%%%%%%%%%%%%%%

\appendix

\section{Appendix}\label{Sec: appendix}

\subsection{Coercivity-type estimates for positive kernels}\label{Subsec: posKer}

\begin{lemma}\label{Lem: AppendixPoincare}
Consider the bilinear form $B_N: X^1_N \times X^1_N \longrightarrow \R$ as given by Definition  \ref{d:bilinearform_N}. Assume further that the kernel $\gamma$ given in the definition of $B_N$ is strictly positive. Then, there exists a positive constant $\beta$, such that
\begin{equation}\label{e:poincare_2}
 \| u\|^2_{L^2(\R^n)} \ \leq \ \beta  B_N[u,u].
 \end{equation}
\end{lemma}

This result and its proof resemble that of \cite[Proposition 1]{mengesha2013analysis} in the context of scalar peridynamic models.

\begin{proof}
To reach a contradiction, assume that there exists a sequence $\{ u_k\}$ in $X^1_N\subset L^2(\R^n)$ with the property that $\| u_k\|_{L^2(\R^n)} =1$ and that as $k \rightarrow \infty$,
\[ | B_N[u_k,u_k] |= \frac{1}{2} \int_{\Om} \int_{\Om} (u_k(y) -u_k(x))^2 \gamma(x,y) \;dy\; dx \rightarrow 0.\]
We want to prove that $\|u_k \|_{L^2(\R^n)} \rightarrow 0 $ as $k \rightarrow \infty$.

First, since the sequence is bounded in $L^2(\R^n)$, it has a weakly convergent subsequence, which we again label
as $\{ u_k\}$. Then, $u_k \rightharpoonup \bar{u}$, for some $\bar{u} \in L^2(\R^n)$. 

In what follows we will use the operator $\mathcal{L} : L^2(\R^n) \rightarrow L^2(\R^n)$ given by
\[ \mathcal{L} \phi \ = \ \int_{\R^n} (\phi(y) - \phi(x))\gamma(x,y) \; dy,\]
and the functionals $J_k:  L^2(\R^n) \rightarrow \R$, which depend on a fixed $u_k \in X^1_N$, defined via
\[ J_k(\phi) = \int_{\R^n} \mathcal{L}( \phi) u_k \;dx.\]

Since the kernel $\gamma$ is exponentially localized, the operator $\mathcal{L}:  L^2(\R^n) \rightarrow L^2(\R^n)$ is bounded. Therefore, 
\[ J_k(\phi) = \int_{\R^n} \mathcal{L}(\phi) u_k \;dx \longrightarrow  \int_{\R^n} \mathcal{L}(\phi) \bar{u} \;dx, \]
%since $\mathcal{L} \phi \mid_{\Om} \in L^2(\Om)$.
At the same time, by the nonlocal form of Green's identity \eqref{Eq: NLGreen},
\[
%J_k(\phi) = &\int_\Omega \mathcal{L}(\phi) u_k\;dx = \int_\R \int_\R (\phi(y) - \phi(x)) \gamma(x,y) \;dy\; u_k(x) \;dx,\\
J_k(\phi) \ = \ -\frac{1}{2}  \int_{\R^n} \int_{\R^n} (\phi(y) - \phi(x)) \gamma(x,y) (u_k(y) - u_k(x)) \;dy\;dx,
\]
for all $\phi \in  L^2(\R^n)$, and consequently  we have that $J_k (\phi)$ converges to
%\[ J_k(\phi) \longrightarrow 
\[-\frac{1}{2}  \int_{\R^n} \int_{\R^n} (\phi(y) - \phi(x)) \gamma(x,y) (\bar{u}(y) - \bar{u}(x)) \;dy\;dx,\]
 as $k \rightarrow \infty$.
 
Our first goal is to show that this last expression is equal to zero for all $\phi \in L^2(\R^n)$.
 Using Cauchy-Schwarz, we obtain
\[ |J_k(\phi)|  \ < \ \frac{1}{2} \left( \iint_{\R^n \times \R^n} (u_k(y) - u_k(x))^2 \gamma(x,y) \;dy\;dx \right)^{1/2}
\left( \iint_{\R^n \times \R^n} (\phi(y) - \phi(x))^2 \gamma(x,y) \;dy\;dx \right)^{1/2},
\]
and we notice that the second integral is bounded in $k$ since it is equal to $B_N[\phi, \phi]$. 
On the other hand, the first integral is equal to $B_N[u_k,u_k]$, which converges to zero thanks to our
assumption on the sequence $\{ u_k\}$. Therefore, $J_k(\phi) \rightarrow 0$.

Now, letting $\phi = \bar{u}$ and keeping in mind that $J_k(\bar{u}) \rightarrow -\frac{1}{2} B_N[\bar{u},\bar{u}]$,
 this last result implies that
\[ B_N[\bar{u},\bar{u}] = \iint_{\Om \times \Om}  (\bar{u}(y) - \bar{u}(x))^2 \gamma(x,y)\;dy \;dx = 0\]
Then, since the kernel $\gamma>0$ we also obtain that $\bar{u}$ is constant in $\R^n$. Since constants are excluded from  
$L^2(\R^n)$, we must have that $\overline{u} = 0$ on $\R^n$.

The last thing we need to show is that the sequence $\{ u_k\}$ converges strongly to $\bar{u}$, and thus arrive at our desired contradiction.

Since for fixed $x$, $f(y):= \gamma(x-y)  \in L^2(\R^n)$ and since the sequence $u_k \rightharpoonup 0$ in $L^2(\R^n)$, it follows that $\gamma \ast u_k(x) \rightarrow 0 $ point-wise and then also in $L^2(\R^n)$. 
As a result,

\begin{align*}
J_k(u_k)\ = \ & \iint_{\R^n \times \R^n} (u_k(y) -u_k(x)) \gamma(x,y) \;dy \; u_k(x) \;dx\\
J_k(u_k) \ = \ & \int_{\R^n} u_k ( \gamma \ast u_k) \;dx - \int_{\R^n} |u_k|^2 \int_{\R^n} \gamma(x,y) \;dy \;dx\\
J_k(u_k) \ = \ &  \int_{\R^n} u_k ( \gamma \ast u_k) \;dx  - \Ga \| u_k\|^2_{L^2(\R^n)}.
\end{align*} 
Taking the limit as $k \rightarrow \infty$
\[   \Ga \| u_k \|^2_{L^2(\R^n)} =|J_k(u_k)| \longrightarrow 0, \]
completing the proof.
%\JMS{I think for this step it's not important that $\La > 0$, just that $\La \neq 0$. Which may be relevant for sign-changing case}
\end{proof}

To prove that the norms $X^1_N$ and $L^2(\Omega)$ are equivalent one first shows that the following exterior problem has a unique solution. This same result is shown in \cite{olena2021, cappanera2024analysis}.  

{\bf Exterior Problem:} Let $g\in L^2(\Omega)$ and consider
\begin{equation}\label{e:exterior}
\begin{split}
\mathcal{L} u \ & = \ 0 \qquad x \in \Omega^c\\
u \ & = \ g \qquad x \in \Omega.
\end{split}
\end{equation}
Define the space 
\[Y_N = \{ u \in L^2(\R^n) \mid \quad B_N[u,u]< \infty, \quad u(x) = 0 \quad x\in \Omega\}\]
with norm given by $\|u\|_{Y_N} = \| u\|_{L^2(\Omega^c)} + |u|_B$ where $|u|^2_B = B_N[u,u]$. 
The weak formulation of the exterior problem then reads
\begin{equation}\label{e:weak_exterior}
B_N[\hat{u},v] \ = \ (\mathcal{L} \tilde{g}, v)_{L^2(\Omega^c)} \quad \forall v \in Y_N.
\end{equation}
where $\hat{u} = u-\tilde{g}$ and $\tilde{g}$ is the extension of $g$ by zero to $\R^n$.

\begin{lemma}\label{l:exteriorproblem}
Let $g \in L^2(\Omega)$ and define $\tilde{g}$ as its extension by zero to $\R^n$. Then, the nonlocal Dirichlet problem \eqref{e:exterior} has a unique weak solution $u \in L^2(\Omega)$, given by $u = \hat{u} + g$, where $\hat{u} \in Y_N$ satisfies \eqref{e:weak_exterior}. Moreover, there is a positive constant $C$, such that
\[ \| u\|_{L^2(\Omega^c)} \ \leq \ C \|g\|_{L^2(\Omega)}.\]
\end{lemma}
A proof of this result is given in \cite[Lemma 3.1]{cappanera2024analysis}  and is a consequence of the Lax-Milgram theorem.
The next result is a consequence of the Poincar{\'e} inequality \eqref{e:poincare_2} and the above lemma.
%\JMS{I think we just state the lemma once, probably in the main body of the text?}
\begin{lemma}\label{l:equivalent_normsN}
There exist positive constants $c$ and $C$ such that for any $u \in X^0_N$,
\[ c\|u\|_{L^2(\Omega)} \ \leq \ \| u\|_{X^0_N} \ \leq \  C\|u\|_{L^2(\Omega)}.\]
\end{lemma}

\bibliographystyle{plain} % We choose the "plain" reference style, can change later
\bibliography{refs}
	
\end{document}